\documentclass[11pt,twoside]{amsart}
\usepackage[all]{xy}
        \usepackage {amssymb,latexsym,amsthm,amsmath,mathtools}
\usepackage{dsfont}
        \usepackage{enumitem,color}
        
\usepackage{longtable}
\usepackage{mathtools}

\long\def\symbolfootnote[#1]#2{\begingroup%
\def\thefootnote{\fnsymbol{footnote}}\footnote[#1]{#2}\endgroup}

\newcommand{\Sp}{\ensuremath{\textup{Sp}(2n,q)}}
\newcommand{\fq}{\ensuremath{\mathbb{F}_q}}
\newcommand{\fqn}{\ensuremath{\mathbb{F}}}
\newcommand{\tr}{\ensuremath{{}^t\!}}

\newcommand{\pa}{\mathcal{P}}
\newcommand{\ssp}{\mathcal{D}}
\newcommand{\Aut}{\textup{Aut}}

\newcommand{\GL}{\textup{GL}}
\newcommand{\SP}{\textup{Sp}}
\newcommand{\On}{\textup{O}}
\newcommand{\e}{\epsilon}
\newcommand{\U}{\textup{U}}

\newcommand{\im}{\textup{im}}

\makeatletter
\def\imod#1{\allowbreak\mkern10mu({\operator@font mod}\,\,#1)}
\makeatother
\makeatletter
\renewcommand*\env@matrix[1][*\c@MaxMatrixCols c]{%
  \hskip -\arraycolsep
  \let\@ifnextchar\new@ifnextchar
  \array{#1}}
\makeatother

\newtheorem{theorem}{Theorem}[section]
\newtheorem{lemma}[theorem]{Lemma}
\newtheorem{corollary}[theorem]{Corollary}
\newtheorem{proposition}[theorem]{Proposition}
\newtheorem*{theorem*}{Theorem}
\theoremstyle{definition}
\newtheorem{ques}[theorem]{Question}
\newtheorem{define}[theorem]{Definition}
\newtheorem{remark}[theorem]{Remark}
\newtheorem{example}[theorem]{Example}
\newtheorem{notation}[theorem]{Notation}
\numberwithin{equation}{section}

\newcommand{\ignore}[1]{}

\newcommand{\mynote}[1]{}
\begin{document}
\setcounter{section}{0}
% document information
\title[Powers in finite orthogonal and symplectic groups]{Powers in finite orthogonal and symplectic groups: A generating function approach}
\author[Panja S.]{Saikat Panja}
\email{panjasaikat300@gmail.com}
\author[Singh A.]{Anupam Singh}
\email{anupamk18@gmail.com}
\address{IISER Pune, Dr. Homi Bhabha Road, Pashan, Pune 411 008, India}
\thanks{The first named author has been supported by NBHM PhD scholarship. The second named author is funded by SERB through CRG/2019/000271 for this research.}
%\date{}
\today

%%%%%%%%%%%%%%%%%%%%%%%%
\setcounter{tocdepth}{1}
\begin{abstract}
For an integer $M\geq 2$ and a finite group $G$, an element $\alpha\in G$ is called an $M$-th power if it satisfies $A^M=\alpha$ for some $A\in G$. In this article, we
will deal with the case when $G$ is finite symplectic or orthogonal group over a field 
of order $q$. We introduce the notion of $M^*$-power SRIM polynomials. This, amalgamated with the concept of $M$-power polynomial, we 
provide the complete classification of the conjugacy classes of regular semisimple, semisimple, cyclic and regular elements in $G$, which are $M$-th powers, when $(M,q)=1$. The approach here is of generating functions, as worked on by Jason Fulman, Peter M. Neumann, and Cheryl Praeger in the memoir ``A generating function
approach to the enumeration of matrices in classical groups over finite fields". As a byproduct, we obtain the corresponding probabilities, in terms of generating functions.
\end{abstract}

\maketitle
%%%%%%%%%%%%%%%%%%%%%%%%
%\tableofcontents
\section{Introduction}
\subsection{Question in the general context} The motivation behind this work dates back to the work of two of the great mathematicians of the last century, A. Borel and E. Waring. Given an element $w\in \mathcal{F}_l$ (the free group on $l$ generators), the map associated with
$w$ by plugging elements of $G^l$ in $w$, is called a \emph{word map}.
It was proved by A. Borel (in \cite{Bo83}) (and later by Larsen independently in \cite{La04}) that given a semisimple algebraic group $\mathcal{G}$ and a word map $w: \mathcal{G}^l\longrightarrow \mathcal{G}$, it is a dominant map. The image of $w$ will be denoted as $w(\mathcal{G})$ hereafter. The result due to Borel reveals the surprising result that $w(\mathcal{G})^2=\mathcal{G}$. On the other hand,
it was E. Waring, who mentioned in his paper ``Meditationes Algebraicae"
that,
``every natural number is a sum of at most 9 cubes;
every natural number is a sum of at most 19 fourth powers;
and so on". These two problems gave rise to the following general question in the context of group theory.
\noindent
\begin{ques}
Given a group $G$ and a word $w$ on $l$ generators, does there exists a $m(w)\in\mathbb{N}$ such that $(w(G))^{m(w)}=G$?
\end{ques}
\noindent 
This question is known as \emph{Waring type problem} in group theory and 
has attracted an ample amount of attention from mathematicians in the past half-century.  Substantial progress has been made and many fundamental questions are solved,
using a wide spectrum of tools, including representation theory, probability, and
geometry. A recent breakthrough in this direction is the affirmation of
 Ore's conjecture (which states that the commutator map corresponding to the word $xyx^{-1}y^{-1}\in F_2$ is surjective in the case of finite non-abelian simple groups), by Liebeck, O'Brien,
 Shalev and Tiep \cite{LiObSh10}, using the methods from character theory. 
 They proved that if $G$ is any quasisimple classical group over a finite field, 
 then every element of $G$ is a commutator, using character-theoretic results due to 
 Frobenius. In \cite{LiObSh12}, the results about the product of squares in the finite non-abelian  simple groups are proved.
 It was proved that 
 every element of a non-abelian finite simple group $G$ is a product of two squares. For a survey of these results and further problems in the context of group theory, we refer the reader to the excellent survey article due to Shalev \cite{Sh09}.

In this article, we will be interested in the map associated with $w=x^M\in\mathcal{F}_1$ where $M\geq 2$ is an integer. This is a part of the ongoing project, where we intend to draw a conclusion about the image size of the power maps for finite groups of Lie type. The complete solution to this for the case of $\textup{GL}(n,q)$ has been described in \cite{ks}. But the existence of a root in the general linear group does not
guarantee the existence of a root in the symplectic or orthogonal group. This has been a great motivation behind this work.
The asymptotics of the powers in finite reductive groups has been pursued in \cite{krs1}. Indeed, the authors therein estimate the proportion of regular semisimple, semisimple and regular which are $M$-th powers in the concerned groups, as $q$ tends to infinity.

We will be giving the exact ratio for the symplectic and orthogonal groups over a finite field $\fq$. We restrict ourselves to the case $(M,q)=1$, as the case $(M,q)\neq 1$ is more intricate and will be followed up 
in future work. 
Our main results are Theorem \ref{012} (and Theorem \ref{025}), Theorem \ref{016} (and Theorem \ref{028}), Theorem \ref{024} (and Theorem \ref{029}, Theorem \ref{031}) concerning generating function for the probability of
a separable, semisimple, cyclic, and regular element respectively to be an $M$-th power in symplectic groups (and orthogonal groups). 
In follow-up work, we will be also looking into the case of power maps in exceptional groups of Lie type. 

\subsection{Methodology}
We take the methods of statistical group theory, where generating functions play a key role. Before describing this, note that if $x\in G$ is an $M$-th power, then so are all conjugates of $x$. Hence to solve the question, it is admirable that we first find the conjugacy classes, which are $M$-th powers. 
The conjugacy class of finite orthogonal and symplectic groups is given by the combinatorial data consisting of self-reciprocal monic polynomials and signed symplectic or orthogonal partitions.
 The first description of conjugacy classes in these groups was discussed in the paper of Wall (see \cite{wa}). The enumeration for the conjugacy classes is done with the machinery of generating functions along with the results of G. Wall.

 Coming back to the viewpoint of statistical group theory, another way of looking into conjugacy classes is via cycle indices, which was introduced by P\'{o}lya for the symmetric group in the paper \cite{po}. This can be briefly 
described as follows. For $\pi\in S_n$, let $a_i(\pi)$ denotes the number of $i$-cycles
in $\pi$. 
Recall that in $S_n$, the number of elements with $a_i$ many $i$-cycles is given by
$n!/\prod_{i=1}^{n}a_i!i^{a_i}$. This along with the Taylor expansion of $e^z$ gives that
\begin{align*}
\displaystyle\sum_{n=0}^{\infty}\dfrac{u^n}{n!}\sum_{\pi\in S_n}^{}\prod_{i}^{}x_i^{a_i(\pi)}=\prod_{m=1}^{\infty}e^{x_mu^m/m}.
\end{align*}	
Since then the concept of cycle index has been developed for various groups and has been used to derive exemplary results. For example in \cite{f1}  cycle indices for the finite classical groups more precisely for unitary, symplectic, and orthogonal groups are studied. In the symplectic and in the orthogonal case it is assumed that $q$ is odd. For the orthogonal groups a mixed cycle index is defined, taking into account both groups $\On^+(n,q)$ and $\On^-(n,q)$ at the same time. In the memoir \cite{fnp}, the authors consider probabilistic properties for classical groups over a fixed finite field of cardinality $q$ when the rank goes to infinity. 
The results are about the asymptotics of corresponding probability.
These results are very important and already have been used in many contexts, including the design of algorithms in group theory, random generation of simple groups, monodromy groups of curves, and derangements. Similar works have been pursued in an enormous amount of texts and a philomath is suggested to look
into 
\cite{br1},
 \cite{br2}, \cite{br3},
\cite{f1}, \cite{f2}, \cite{fst} to have better understanding of this direction. 
We will be using these cycle indices for finite orthogonal and
symplectic groups, drawn from \cite{f2} and some special polynomials to conclude our result. For reasons coming from the theory of cycle indices, the results for orthogonal groups rely on the
results about symplectic groups.

\subsection{Organization of the paper}
Throughout the article $q=p^a$, a power of a prime. In Section $2$, we will be recalling the description of the central objects of this paper, the finite orthogonal and symplectic groups.
The key role in calculating the generating functions 
are played by the conjugacy classes and centralizers of these groups, which will be discussed in Section $3$. 
In Section $4$, we introduce the notion of $M^*$ polynomials and count the number of such polynomials over a finite field. This section has been coupled with examples to have a better understanding of the proposed definitions.
Section $5,6,7$ and $8$ are devoted to calculating generating functions for separable, semisimple, cyclic, and regular matrices in symplectic and orthogonal groups respectively.
This article is based on a part of Ph.D. thesis of the first named author. A more elaborate versions of the results can be found in \cite{p22}.

%Let $G$ be a finite group and $\omega \colon x\mapsto x^M$ be the power map on $G$. We would like 
%to determine when an element of $G$ is in the set $\omega(G)$ and use this to enumerate the set 
%$\omega(G)$, in the sense of getting generating functions. In~\cite{ks}, this 
%question is answered for the group $\GL(n,q)$ for regular, regular semisimple and semisimple 
%elements as well as conjugacy classes. In this article, we extend the work done there for the 
%finite symplectic and orthogonal groups. Note that the case $M=1$ is work of Fulman, Wall,...... 
%Thus, in what follows we assume $M\geq 2$. 
%
%
%%%%%%%%%%%%%%%%%%%%%%%%%%%%%%%
\subsection*{Acknowledgment}
The authors thank B. Sury and Amit Kulshrestha for their interest in this work.

%%%%%%%%%%%%%%%%%%%%%%%%%%%%%%%%%%%%
\section{Orthogonal and Symplectic Groups}
In this section, we briefly recall the orthogonal and symplectic groups over a finite field 
$\mathbb F_q$, mainly to set the notation for the rest of the paper. The treatment here closely follows that of
\cite{fnp}, \cite{wa} and \cite{bg}.

\subsection{Orthogonal groups} Let $V$ be an $m$-dimensional vector space over a finite field $\mathbb F_q$. Then there are at most two non-equivalent non-degenerate quadratic forms on $V$. The orthogonal group consists of elements of $\GL(V)$ which preserve a non-degenerate quadratic form $Q$. 

When $m=2n$ for some $n\geq 1$, up to equivalence there are two such forms denoted as $Q^+$ and $Q^-$. These are as follows. Fix $a\in \mathbb F_q$ such that $t^2 + t + a \in \fq[t]$ is irreducible. Then the two non-equivalent forms are given by
\begin{enumerate}
\item $Q^+(x_1, \ldots, x_m) = x_{1}x_{2} + x_3x_4+\cdots + x_{2n-1}x_{2n}$, and
\item $Q^-(x_1,\ldots, x_m)= x_1^2 + x_1x_2+ax_2^2 + x_3x_4+\cdots + x_{2n-1}x_{2n}$.
\end{enumerate}
The orthogonal group preserving $Q^+$ will be denoted as $\On^+(m, q)$, and the orthogonal group preserving $Q^-$ will be denoted as $\On^-(m, q)$.

When $m=2n+1$, for $q$ even there is only one (up to equivalence) quadratic form, namely 
$Q(x_1, \ldots, x_m)=x_1^2+\sum\limits_{i=1}^n x_{2i}x_{2i+1}$ and hence there is only one (up to conjugacy) orthogonal group. If $q$ is odd, then up to equivalence, there are two non-degenerate quadratic forms. But, these two forms give isomorphic orthogonal groups. We take $Q(x_1, \ldots, x_m)= x_1^2+\cdots + x_m^2$. Thus, in case $m=2n+1$, up to conjugacy, we have only one orthogonal group. This will be denoted as $\On^0(m, q)$. 

As it is common in literature, we will use the notation $\On^\epsilon(m, q)$ to denote any of the orthogonal group above where $\epsilon \in \{0, +, -\}$. With respect to an appropriate basis, we will fix the matrices of the symmetric bilinear forms (associated to the quadratic forms $Q^{\epsilon}$) as follows: 
$$J_{0}=\begin{pmatrix} 0 & 0 & \Lambda_n \\ 0 & \alpha & 0\\ \Lambda_n & 0 & 0 
\end{pmatrix},  J_{+}=\begin{pmatrix} 0& \Lambda_n\\ \Lambda_n & 0
\end{pmatrix}, J_{-}=\begin{pmatrix} 0 & 0 & 0 & \Lambda_{n-1}\\
0 & 1 & 0 & 0 \\ 0 & 0 & -\delta & 0 \\ \Lambda_{n-1}& 0 & 0 & 0
\end{pmatrix}$$ 
where $\alpha\in\fq^\times,\delta\in\fq\setminus\fq^2$, and $\Lambda_l =\begin{pmatrix}
0 & 0 & \cdots & 0 & 1\\ 0 & 0 & \cdots & 1 & 0\\
\vdots & \vdots & \reflectbox{$\ddots$} & \vdots & \vdots\\
1 & 0 & \cdots & 0 & 0 \end{pmatrix}$, an $l\times l$ matrix. Then, the orthogonal group in matrix form is 
$$\On^\e(m, q) = \{A \in \GL(m,q)\mid \tr{A}J_{\e}A=J_{\e}\}.$$
Adapting the notations of~\cite{fnp}, we define the type of an orthogonal space as follows.
\begin{define} The type of an orthogonal space $(V,Q)$ of dimension $m$ is
\begin{align*}
\tau(V)=\begin{cases}
1 & \text{if } m \text{ is odd, }q\equiv 1\imod 4, Q\sim\sum x_i^2,\\
-1 & \text{if } m \text{ is odd, }q \equiv 1\imod 4, Q\sim b\sum x_i^2,\\
\iota^m & \text{if } m \text{ is odd, }q\equiv 3\imod 4, Q\sim\sum x_i^2,\\
(-\iota)^m & \text{if } m \text{ is odd, }q\equiv 3\imod 4, Q\sim b\sum x_i^2,
\end{cases}
\end{align*}
where $\iota\in\mathbb{C}$ satisfies $\iota^2=-1$, and $b\in\fq\setminus\fq^2$.
\end{define}
\noindent More generally, when $V$ is orthogonal direct sum $V_1\oplus V_2 \oplus \cdots \oplus V_l$, the type is defined by $\tau(V)=\prod\limits_{i=1}^l\tau(V_i)$.

%%%%%%%%%%%%
\subsection{Symplectic group}
Let $V$ be a vector space of dimension $2n$ over $\mathbb F_q$. There is a unique non-degenerate alternating bilinear form on $V$. We consider the form given by 
$$\left<(x_i)_{i=1}^{2n}, (y_j)_{j=1}^{2n} \right> = \sum_{j=1}^{n}x_jy_{2n+1-j}-\sum_{i=0 } ^{n-1}x_{2n-i}y_{i+1}.$$ 
The symplectic group is the subgroup of $\GL(V)$ consisting of those elements which preserve this alternating form on $V$. By fixing an appropriate basis, the matrix of the form is $J= \begin{pmatrix} 0 & \Lambda_n\\ -\Lambda_n & 0 \end{pmatrix}$ where 
$\Lambda_n=\begin{pmatrix} 0 & 0 & \cdots & 0 & 1\\
0 & 0 & \cdots & 1 & 0\\ \vdots & \vdots & \reflectbox{$\ddots$} & \vdots & \vdots\\
1 & 0 & \cdots & 0 & 0 \end{pmatrix}$ and 
$$\Sp =\{A \in \GL(2n,q) \mid \tr AJA=J\}.$$ 
Since all alternating forms are equivalent over $\mathbb F_q$, the symplectic groups obtained with respect to different forms are conjugate within $\GL(2n, q)$. 

%%%%%%%%%%%%%%%%%%%
\subsection{Main Question}
We recall some definitions here. Let $G=\Sp$ or $\On^\epsilon(m, q)$ defined as above where $m=2n$ or $2n+1$ depending on $m$ is even or odd. Then, $n$ is the Lie rank of $G$. An element $A\in G$ is said to be 
\begin{enumerate}
\item \textbf{separable} if the characteristic polynomial of $A$ is separable over $\mathbb F_q$,
\item \textbf{semisimple} if the minimal polynomial of $A$ is separable over $\mathbb F_q$,
\item \textbf{cyclic} if the minimal polynomial of $A$ is same as the characteristic polynomial of $A$, and,  
\item  \textbf{regular} if the centraliser of $A$ in $G$ has dimension equal to the Lie rank of $G$.
\end{enumerate}
We say an element $g\in G$ is an $M^{th}$ power in $G$ or it has $M^{th}$-root in $G$ if the equation $X^M=g$ has a solution. As in~\cite{ks}, we would like to study if separable, semsimple, cyclic or regular elements are $M^{th}$ powers. More precisely, we are interested in the generating functions of the following quantities for these groups:
\begin{enumerate}
\item $cs^M_G(n, q)$, $css^M_G(n, q)$, $cc^M_G(n, q)$ and $cr^M_G(n, q)$ which denote the ratio of the number of separable, semisimple, cyclic and regular conjugacy class, respectively, of $G$ which are $M$-th power, to the number of all conjugacy classes.
\item $s^M_G(n, q)$, $ss^M_G(n, q)$, $c^M_G(n, q)$ and $r^M_G(n, q)$ which denote the ratio of the number of separable, semisimple, cyclic and regular elements in $G$ which are $M$-th power, with respect to the $|G|$.
\end{enumerate}
Further, the associated generating functions will be defined in the later sections accordingly.
The main question here is to determine these generating functions using the canonical forms of elements in these groups.

\section{Conjugacy classes in Orthogonal and Symplectic group}
As mentioned in the introduction, if an element is an $M$-th power, so is the conjugacy class containing that element. Thus, it is necessary that we know the conjugacy classes in detail for these groups.
This is achieved using the combinatorial data consisting of monic polynomials, and 
signed partitions attached to these polynomials. This is the classic work of \cite{wa}.
We recall briefly the results therein, which will be used further.
\begin{define}
A \textbf{symplectic signed partition} is a partition of a number $k$, such that the odd parts 
have even multiplicity and even parts have a sign associated with it. The set
of all symplectic signed partitions will be denoted as $\ssp_{\SP}$.
\end{define}
\begin{define}
An \textbf{orthogonal signed partition} is a partition of a number $k$, such that all even parts have even multiplicity, and all odd parts have a sign associated with it. The set
of all orthogonal signed partitions will be denoted as $\ssp_{\On}$.
\end{define}
\begin{example}\begin{enumerate}
\item The partition $6^{+2} 3^4 2^{-3} 1^2$ is a symplectic signed partition of $32$.
\item The partition $7^{+2}7^{-2} 3^{+3} 2^6 1^{-2}$ is an orthogonal signed partition of $51$.
\end{enumerate}
\end{example}
\begin{define}
The \textit{dual} of a monic degree $r$ polynomial $f(x)\in k[x]$ satisfying $f(0)\neq 0$, is the polynomial given by $f^*(x)=f(0)^{-1}x^rf(x^{-1})$. 
The polynomial $f$ will be called \textit{$*$-symmetric} (or \emph{self reciprocal}) if $f=f^*$. A monic polynomial $f(x)\in \fq[x]$, will be called to be \textbf{$*-$irreducible}
 if and only if it does not have any proper self-reciprocal factor.
\end{define}

It can be shown that characteristic polynomial of symplectic or orthogonal matrix is self reciprocal. Indeed if $\lambda$ is a root of the characteristic 
polynomial of a symplectic (or orthogonal) matrix, so is $\lambda^{-1}$. We follow 
J. Milnor's terminology \cite{mi} to distinguish between the $*$-irreducible factors of the characteristic polynomials. We call a $*$-irreducible polynomial $f$ to be

\noindent(1) Type $1$\index{Type $1$ polynomial} if $f=f^*$ and $f$ is irreducible polynomial of even degree;\

\noindent(2) Type $2$\index{Type $2$ polynomial} if $f=gg^*$ and $g$ is irreducible polynomial satisfying $g\neq g^*$;\

\noindent(3) Type $3$\index{Type $3$ polynomial} if $f(x)=x\pm 1$.\

According to \cite{wa}, \cite{sh}, the conjugacy classes of $\Sp$ are parameterized by the functions 
$\lambda:\Phi\rightarrow\pa^{2n}\cup\ssp_{\SP}^{2n}$, where $\Phi$ denotes the set of all
 monic,  non-constant, irreducible polynomials, $\pa^{2n}$ is
 the set of all partitions 
of $1\leq k\leq 2n$ and $\ssp_{\SP}^{2n}$ is the set of all symplectic
 signed partitions 
of $1\leq k\leq 2n$. Such a $\lambda$ represent a conjugacy class of $\Sp$
if and only if \begin{enumerate}
\itemindent=-13pt
\item $\lambda(x)=0$,
\item $\lambda_{\varphi^*}=\lambda_\varphi$,
\item $\lambda_\varphi\in\ssp^n_{\SP}$ iff $\varphi=x\pm 1$ (we distinguish this $\lambda$, by denoting it $\lambda^\pm$),
\item $\displaystyle\sum_{\varphi}|\lambda_\varphi|\textup{deg}(\varphi)=2n$.
\end{enumerate}
Also from \cite{wa}, \cite{sh}, we find out that similar kind of statement is true for the groups $\On^\e(n,\fq)$. The conjugacy classes of $\On^\e(n,\fq)$ are parameterized by the functions 
$\lambda:\Phi\rightarrow\pa^{n}\cup\ssp_{\On}^{n}$, where $\Phi$ denotes the set of all
 monic,  non-constant, irreducible polynomials, $\pa^{n}$ is
 the set of all partitions 
of $1\leq k\leq n$ and $\ssp_{\On}^{n}$ is the set of all symplectic
 signed partitions 
of $1\leq k\leq n$. Such a $\lambda$ represent a conjugacy class of $\Sp$
if and only if \begin{enumerate}
\itemindent=-13pt
\item $\lambda(x)=0$,
\item $\lambda_{\varphi^*}=\lambda_\varphi$,
\item $\lambda_\varphi\in\ssp^n_{\On}$ iff $\varphi=x\pm 1$ (we distinguish this $\lambda$, by denoting it $\lambda^\pm$),
\item $\displaystyle\sum_{\varphi}|\lambda_\varphi|\textup{deg}(\varphi)=n$.
\end{enumerate}
Class representative corresponding to given data can be found in \cite{ta1}, \cite{ta2}, \cite{glo} 
and we will mention them whenever needed. We mention the
following results about the conjugacy class size (and hence the size of the centraliser)
 of elements corresponding to given data
, which can be found in \cite{wa}.
\begin{lemma}[\cite{wa}, pp. $36$]\label{002}
Let $X\in\Sp$ be a matrix corresponding to the data $\Delta_X=\{(\phi,\mu_\phi)
:\phi\in\Phi_X\subset \Phi\}$. Then the conjugacy class of
$X$ in $\Sp$ is of size $\dfrac{|\Sp|}{\prod\limits_{\phi}B(\phi)}$
where $B(\phi)$ and $A(\phi^\mu)$ are defined as follows
\begin{align*}
A(\phi^\mu)=\begin{cases}
|\U(m_\mu,Q)| & \text{if } \phi(x)=\phi^*(x)\neq x\pm 1\\
|\GL(m_\mu,Q)|^\frac{1}{2}& \text{if } \phi\neq\phi^*\\
|\SP(m_\mu,q)| &\text{if }\phi(x)=x\pm 1,~\mu\text{ odd}\\
|q^{\frac{1}{2}m_\mu}\On^\e(m_\mu,q)|&\text{if }\phi(x)=x\pm 1,~\mu\text{ even}
\end{cases},
\end{align*}
where $\e$ gets determined by the sign of the corresponding partition, 
$Q=q^{|\phi|}$, $m_\mu=m(\phi^\mu)$ and
\begin{align*}
B(\phi) = Q^{\sum\limits_{\mu<\nu}\mu m_\mu m_\nu+
\frac{1}{2}\sum\limits_{\mu}(\mu-1)m_\mu^2}\prod\limits_{\mu}A(\phi^\mu).
\end{align*}
\end{lemma}
\begin{lemma}[\cite{wa}, pp. $39$]\label{003}
Let $X\in\On^\e(n,q)$ be a matrix corresponding to the data $\Delta_X=\{(\phi,\mu_\phi)
:\phi\in\Phi_X\subset \Phi\}$. Then the conjugacy class of
$X$ in $\On^\e(n,q)$ is of size $\dfrac{|\Sp|}{\prod\limits_{\phi}B(\phi)}$
where $B(\phi)$ and $A(\phi^\mu)$ are defined as before, except when $\phi(x)=x\pm 1$,
\begin{align*}
A(\phi^\mu)=\begin{cases}
|\On^{\e'}(m_\mu,q)|&\text{if }\mu\text{ odd}\\
q^{-\frac{1}{2}m_{\mu}}|\SP(m_\mu,q)|&\text{if }\mu\text{ even}
\end{cases},
\end{align*}
where $\e'$ in $\On^{\e'}(m_\mu,q)$  gets determined by the 
corresponding sign of the part, of the partition.
\end{lemma}
With all the basic tools now in place, we are ready to 
move to the next section, where
we detemine when a matrix $A$ whose characteristic polynomial
is a $*$-irreducible polynomial of type $1$ or $2$, is an $M$-th power. 
This information is further used in subsequent chapters to determine the 
desired generating functions, with the help of the concept of 
central join of two matrices, following \cite{ta1}, \cite{ta2}.

\section{$M$-power and $M^*$-power polynomial}
In \cite{ks}, the concept of $M$-power polynomial has been introduced, which plays a crucial role of
identifying matrices of $\GL(n,q)$ which are $M$-th powers, in terms of the combinatorial data. 
We will be encountering another kind of polynomials, which are irreducible factors of characteristic 
polynomials for identifying matrices in $\SP(2n,q)$, $\On^\epsilon(n,q)$. 
\subsection{Special polynomials}
\begin{lemma}\label{001}
\begin{enumerate}
\itemindent=-13pt
\item Each Self reciprocal irreducible monic (SRIM) 
polynomial of degree $2n~ (n \geq 1)$ over $\fq$ is a factor of the
polynomial
\begin{align} 
H_{q,n}(x)\coloneqq x^{q^n+1}-1\in\fq[x]
,
\end{align}
\item Each irreducible factor of degree  $\geq 2$ of $H_{q,n}(x)$ is a 
SRIM-polynomial of
degree $2d$, where $d$ divides $n$ such that $\frac{n}{d}$ is odd.
\end{enumerate}
\end{lemma}
\begin{example}
\begin{enumerate}
\itemindent=-13pt
\item The SRIM polynomials of degree $4$ over $\fqn_5$, are factors of 
$x^{26}-1\in\fqn_5[x]$. Using \cite{sage}, it can be found out that in $\fqn_5[x]$, we have 
$x^{26}-1=(x + 1)
 (x + 4)
 (x^4 + x^3 + 4x^2 + x + 1)
 (x^4 + 2x^3 + 2x + 1)
 (x^4 + 2x^3 + x^2 + 2x + 1)
 (x^4 + 3x^3 + 3x + 1)
 (x^4 + 3x^3 + x^2 + 3x + 1)
 (x^4 + 4x^3 + 4x^2 + 4x + 1).$ This gives all the SRIM polynomial of degree $4$ over $\fqn_5$.
\item Also using \cite{sage}, we have that in $\fqn_2[x]$ the polynomial 
$x^{2^6+1}-1$ factorizes as $(x + 1)
 (x^4 + x^3 + x^2 + x + 1)
 (x^{12} + x^8 + x^7 + x^6 + x^5 + x^4 + 1)
 (x^{12} + x^{10} + x^7 + x^6 + x^5 + x^2 + 1)
 (x^{12} + x^{10} + x^9 + x^8 + x^6 + x^4 + x^3 + x^2 + 1)
 (x^{12} + x^{11} + x^9 + x^7 + x^6 + x^5 + x^3 + x + 1)
 (x^{12} + x^{11} + x^{10} + x^9 + x^8 + x^7 + x^6 + x^5 + x^4 + x^3 + x^2 + x + 1).$ This gives all the SRIM polynomial of degree $12$ over $\fqn_2$.
\end{enumerate}
\end{example}
\begin{define}
A SRIM polynomial $f\in\fq[x]$ of degree $2k$, $k\geq 1$, is said to 
be an \textbf{$M^*$-power SRIM polynomial} if and only if $f(x^M)$ 
has a SRIM factor $g\in\fq[x]$, of degree $2k$. Denote the set of 
$M^*$-power SRIM polynomial (of degree $\geq 2$) by $\Phi_M^*$.
\end{define}
\begin{define}\cite{ks}
A monic irreducible polynomial $f\in\fq[x]$ of degree $k$, $k\geq 1$, is said to 
be an \textbf{$M$-power polynomial} if and only if $f(x^M)$ 
has a monic irreducible factor $g\in\fq[x]$, of degree $k$.  Denote the set of 
$M$-power polynomial ($\neq x$) by $\Phi_M$.
\end{define}
\begin{example}\label{2-but-not-2*}
(a) Consider $\fqn_5$ and the polynomial $x^4 + 3x^3 + 3x + 1
\in\fqn_5[x]$. Then $x^8 + 3x^6  + 3x^2 + 1= (x^4 + 2x^3 + x^2 + 2x + 1)
(x^4 + 3x^3 + x^2 + 3x + 1)$. Hence $x^4 + 3x^3 + 3x + 1$ is a $2^*$-power SRIM polynomial.

\noindent(b) Consider $\fqn_5$ and the polynomial $x^4 + 3x^3 + 1x^2 + 3x + 1
\in\fqn_5[x]$. Then $x^8 + 3x^6 + x^4 + 3x^2 + 1=(x^4 + 2x^3 + x^2 + 3x + 1)(x^4 + 3x^3 + x^2 + 2x + 1)$. Thus it is a $2$-power polynomial but not a $2^*$-power SRIM polynomial.
\end{example}
\begin{proposition}\label{007}
Let $N^*_M(q,2k)$ denotes the number of $M^*$-power SRIM polynomial 
of degree $2k$, $k\geq 1$. Then \begin{equation}
N^*_M(q,2k)=\dfrac{1}{2k(M,q^{2k}-1)}\displaystyle
\sum\limits_{\substack{l=\textup{odd}\\l|2k}}\mu(l)(M(q^{2k/l}-1),q^k+1).
\end{equation}
\end{proposition}
\begin{proof}
    Let $f$ be an $M^*$-power SRIM polynomial of degree $2k$. Then 
$f(x^M)$ has a SRIM factor $g$ of degree $2k$. 
Consider $f,g\in\fqn_{q^{2k}}$.
Then $f=\displaystyle\prod_{i=1}^{2k}(x-\alpha_i),
g=\displaystyle\prod_{i=1}^{2k}(x-\beta_i)$. As discussed before, 
without loss of generality we may assume that $\beta_i^M=\alpha_i$. 
Considering the map $\theta_M:\fqn_{q^{2k}}\rightarrow\fqn_{q^{2k}}$, we have $\alpha_i\in\im(\theta_M)$, for 
all $i$. Since 
$f$ is SRIM, using Lemma \ref{001} we have that 
%$\alpha_i^{-1}$ is also a root for all $i$. 
%Hence the assignment $\alpha_i\mapsto\alpha_i^{-1}$, extends to an 
%involution of $\fqn_{q^{2k}}$. Since the only non-trivial involution in $
%\textup{Gal}(\fqn_{q^{2k}}/\fq)$ is the map $\chi(\gamma)=\gamma^{q^k}$,
%we have that 
$\alpha_i^{q^k+1}=1$ for all $i$. Thus $\beta_i$ for all $i$, satisfies 
$\beta_i^{M(q^{2k}-1)}=\beta_i^{q^k+1}=1$ and $\beta_i^M=\alpha_i$ 
generates $\fqn_{q^{2k}}$ over $\fq$.

Conversely suppose $\alpha$ satisfies that,
$\alpha^{q^k+1}=1$ and 
generates $\fqn_{q^{2k}}$ over $\fq$. If $\varphi$ is the monic minimal polynomial of $\alpha$, then $\varphi$ is of degree $2k$. Also if 
$\eta$ is any root of $\varphi$, then $\eta=\alpha^{q^l}$, for some $l$,
whence $\eta^{q^k+1}=1$. Thus $\varphi$ is SRIM. So, if
 $N^*_M(q,2k)$ denotes the number of $M^*$-power SRIM polynomial of degree $2k$, then
\begin{equation}
N^*_M(q,2k)=\dfrac{1}{2k}\displaystyle|\{\alpha\in\fqn_{q^{2k}}:
\alpha^{q^k+1}=1,\alpha=\theta_M(\eta) \text{ for some }\eta\in
\fqn_{q^{2k}},\fqn_{q^{2k}}=\fq(\alpha)\}|,
\end{equation} as sets of roots, of distinct irreducible polynomials, are disjoint.
Since $|\theta_M^{-1}(1)|=(M,q^{2k}-1)$, we have that,
\begin{equation*}
N^*_M(q,2k)=\dfrac{1}{2k(M,q^{2k}-1)}\displaystyle|\{\alpha\in\fqn_{q^{2k}}:
\alpha^{q^k+1}=1,\fqn_{q^{2k}}=\fq(\alpha^M)\}|.
\end{equation*}
To ensure $\fqn_{q^{2k}}=\fq(\alpha^M)$, we should have that $
\alpha^M\notin\fqn_{q^l}$ for any $l|2k,l>1$. Since $\alpha^{q^k+1}
=1$,
 we have that $\alpha^{q^{\frac{k}{l}}+1}$ if and only if $l$ is odd
(because $x^m+1$ divides $x^n+1$ if and only if $\frac{n}{m}$ is 
odd). 
Thus $\alpha^M\in\fqn_{q^{2k/l}}$ if and only if $l$ is odd.
For $l$ odd, define 
$E_l=\{\alpha\in\fqn_{q^{2k}}:
\alpha^{q^k+1}=1,\fqn_{q^{2k/l}}=\fq(\alpha^M)\}$. Then 
$|E_l|=(M(q^{2k/l}-1),q^k+1)$, whence by inclusion-exclusion principle the proof is done.

\end{proof}
%\begin{equation*}
%N^*_M(q,2k)=\dfrac{1}{2k(M,q^{2k}-1)}\displaystyle
%\sum\limits_{\substack{l=\textup{odd}\\l|2k}}\mu(l)(M(q^{2k/l-1}),q^k+1),
%\end{equation*}
%Thus we have proved the following
This settles down the case, when a single block is an $M$-power. Now let us proceed for the case when there are more than one block of same type.
\begin{example} We can show that, if
$A$ is a matrix corresponding to the conjugacy class data 
$(x^{12}+2x^{11}+2x^{10}+2x^9+x^8+x^6+x^4+2x^3+2x^2+2x+1,1)$ in $\SP(12,\fqn_3)$, then $A^{73}$ has conjugacy class
data $(x^4 + x^3 + x^2 + x + 1, 1^3)$. 
\end{example}
Now we consider the case when $A$ has more than one block of 
type $1$ but is an $M$-th power of some $\alpha$. 
Since we are interested in the image of the map $x\mapsto x^M$,
we will be considering the case when any $M$-th root of $A$ 
(if exists) has 
single Jordan block of type $1$. Thus if 
minimal polynomial of $A$ (of degree $\frac{2n}{k}$ for some odd $k$), has
 root $\gamma$, we must have that $M$-th root of
$\gamma$ must exist in $\fqn_{q^{2n}}$ and not in any proper 
subfield of $\fqn_{q^{2n}}$. Thus we want to calculate 
the number of SRIM polynomials of degree $\frac{2n}{k}$, over $\mathbb{F}_q$ such that 
if $f(\alpha)=0$ for some $\alpha\in\mathbb{F}_{q^{\frac{2n}{k}}}$, then there exists
$\beta\in\mathbb{F}_{q^{2n}}$ such that 
$\textup{min}_{\mathbb{F}_q}(\beta)$ is SRIM polynomial of order 
$2n$. 
Let $N^*_M(q,2n,\frac{2n}{k})$ denotes the number of SRIM polynomial of degree $\frac{2n}{k}$ such that if $f(\alpha)=0$ for some 
$\alpha\in\fqn_{q^{\frac{2n}{k}}}$ then any $M$-th root of $\alpha$, 
 say $\beta$ lies in $\fqn_{q^{2n}}$ with the property that 
$\fqn_{q^{2n}}=\fq[\beta]$ and $\beta^{q^n+1}=1$.
\begin{proposition}\label{008}
 We have $N^*_M(q,2n,\frac{2n}{k})$ to be equal to
 \begin{align}
\dfrac{1}{2k}\sum\limits_{\substack{s < k\\ s =\text{odd},s|k}}\mu(s)
\dfrac{1}{(M,q^{\frac{2n}{s}}-1)}\left(
\sum\limits_{\substack{l = \text{odd}\\l|\frac{2n}{k}}}
\mu\left(l\right)\left(M\left(q^\frac{n}{kls}
+1\right),q^{\frac{n}{s}}+1\right)\right).
\end{align}
\end{proposition}
\begin{proof}
For $k$ odd and $k|2n$, consider the set 
\begin{equation*}
E_{2n,\frac{2n}{k}}=\left\{\alpha\in\fqn_{\frac{2n}{k}}|\alpha^{\frac{n}{k}+1}
=1,\alpha=\beta^M,\beta\in\fqn_{q^{2n}},
\beta^{q^n+1}=1,
[\fq(\alpha):\fq]=\frac{2n}{k}\right\}.
\end{equation*} 
To enumerate this set let us find the number of $\beta\in\fqn_{q^{2n}}$, 
such that $\beta^M\in E_{2n,\frac{2n}{k}}$. Then $\beta $ satisfies the 
equations $\beta^{q^n+1}=1$, $\beta^{M(\frac{n}{k}+1)}=1$. 
Number of $\beta$ satisfiying these two equations is given by $(M(q^\frac{n}{k}+1),q^n+1)$.
But we should have that $[\fq(\beta^M):\fq]=\frac{2n}{k}$. Hence 
$\beta^M\not\in\fqn_{q^\frac{2n}{kl}}$, $l>1$ being odd . Hence
by inclusion-exclusion principle, the number 
of $\beta\in\fqn_{q^{2n}}$, such that $\beta^M\in E_{2n,\frac{2n}{k}}$ is
$
\sum\limits_{\substack{l = \text{odd}\\l|\frac{2n}{k}}}
\mu\left(l\right)\left(M\left(q^\frac{n}{kl}
+1\right),q^n+1\right).
$
Since $|\theta_M^{-1}(1)|=(M,q^{2n}-1)$ 
where $\theta_M:\fqn_{q^{2n}}\rightarrow\fqn_{q^{2n}}$ is the map 
$\theta_M(x)=x^M$, we have that
\begin{equation*}
|E_{2n,\frac{2n}{k}}|=\dfrac{1}{(M,q^{2n}-1)}
\sum\limits_{\substack{l = \text{odd}\\l|\frac{2n}{k}}}
\mu\left(l\right)\left(M\left(q^\frac{n}{kl}
+1\right),q^n+1\right).
\end{equation*}

Now we want to consider only those $\alpha\in E_{2n,\frac{2n}{k}}$ such that 
it doesn't have any $M$-th root in any proper subfield of 
$\fqn_{q^{2n}}$. Since an $M$-th root, say $\beta$, also has minimal polynomial to be 
SRIM (by hypothesis), we have that $\beta\in\fqn_{q^{\frac{2n}{s}}}$ 
if and only if $s$ is odd.
Hence $\beta\in E_{2n,\frac{2n}{k}}\setminus\bigcup
\limits_{\substack{s < k\\ s =\text{odd},\frac{2n}{k}|\frac{2n}{s}}}E_{2n/s,\frac{2n}{k}}$.
Thus we have that
\begin{equation*}
N^*_M\left(q,2n,\frac{2n}{k}\right)= \dfrac{1}{2k}\sum\limits_{\substack{s < k\\ s =\text{odd},\frac{2n}{k}|\frac{2n}{s}}}\mu(s)
\dfrac{1}{(M,q^{\frac{2n}{s}}-1)}
\sum\limits_{\substack{l = \text{odd}\\l|\frac{2n}{k}}}
\mu\left(l\right)\left(M\left(q^{\frac{n}{kls}}
+1\right),q^{\frac{n}{s}}+1\right),
\end{equation*}
since the sets of roots of irreducible polynomials are disjoint.
\end{proof}
\begin{define}
For a divisor $k$ of $n$, we will call a polynomial $f(x)\in\fq[x]$ 
of degree 
$\frac{n}{k}$ which is not an $M$-power polynomial, to be \textbf{degenerate $(M,n,\frac{n}{k})$ polynomial} 
if and only if minimal polynomial of $\beta$ over $\fq$ is of degree $n$,
where $f(\beta^M)=0$.
 Denote the set of
degenerate $(M,n,\frac{n}{k})$ polynomials ($\neq x$) by $\Phi^u_{M,n,\frac{n}{k}}$. 
Denote by $\Phi^{*,u}_{M,n,\frac{n}{k}}$ the subset of
SRIM polynomials having same property.
\end{define}
\begin{remark}
The quantity $N^*_M(q,2n,\frac{2n}{k})$ counts the number of 
degenerate $(M,2n,\frac{2n}{k})$ SRIM polynomials over $\fq$.
\end{remark}
\begin{remark}
We have that $N_M^*(q,2r)=N_M^*(q,2r,2r)$.
\end{remark}
In case a polynomial is degenerate $(M,n,\frac{n}{k})$ polynomial,
there are $M$-th roots of $\alpha$, where $f(\alpha)=0$, which lies in 
$\fqn_{q^{n}}$ and not in any proper subfield of it. But there might be 
other $M$-th roots which lie in other extensions, as illustarted by the following 
examples.
\begin{example} 
Using \cite{sage}, we have that 
$x^{132} + 2x^{77} + x^{66} + 2x^{55} + 1=(x^{12} + x^{11} + x^{10} + x^9 + 2x^6 + x^3 + x^2 + x + 1)
 (x^{60} + x^{58} + 2x^{57} + 2x^{56} + 2x^{55} + 2x^{54} + 2x^{53} + x^{51} + x^{49} + x^{48} + 2x^{46} + x^{45} + x^{44} + 2x^{43} + 2x^{42} + x^{41} + x^{40} + x^{39} + x^{38} + 2x^{36} + x^{34} + x^{32} + 2x^{31} + x^{30} + x^{27} + x^{26} + 2x^{25} + x^{23} + 2x^{21} + 2x^{19} + x^{17} + x^{16} + x^{15} + x^{13} + 2x^{11} + 2x^7 + 2x^5 + 2x^4 + 2x^3 + 2x^2 + 2x + 1)
 (x^{60} + 2x^{59} + 2x^{58} + 2x^{57} + 2x^{56} + 2x^{55} + 2x^{53} + 2x^{49} + x^{47} + x^{45} + x^{44} + x^{43} + 2x^{41} + 2x^{39} + x^{37} + 2x^{35} + x^{34} + x^{33} + x^{30} + 2x^{29} + x^{28} + x^{26} + 2x^{24} + x^{22} + x^{21} + x^{20} + x^{19} + 2x^{18} + 2x^{17} + x^{16} + x^{15} + 2x^{14} + x^{12} + x^{11} + x^9 + 2x^7 + 2x^6 + 2x^5 + 2x^4 + 2x^3 + x^2 + 1)$ in $\fqn_3[x]$.
Hence note that a root of $x^{12}+2x^{11}+x^{6}+2x^5+1$ has $11$-th root in different degree field extensions.
%(b) Using \cite{sage}, we have that
%$x^{88}+2x^{66}+x^{44}+2x^{22}+1=
%(x^4 + x^3 + 2x + 1)
% (x^4 + 2x^3 + x + 1)
 %(x^{20} + 2x^{18} + x^{17} + 2x^{16} + 2x^{15} + 2x^{12} + 2x^{11} + 2x^{10} + 2x^9 + 2x^7 + x^6 + x^5 + 2x^2 + 2x + 1)
 %(x^{20} + 2x^{18} + 2x^{17} + 2x^{16} + x^{15} + 2x^{12} + x^{11} + 2x^{10} + x^9 + x^7 + x^6 + 2x^5 + 2x^2 + x + 1)
 %(x^{20} + x^{19} + 2x^{18} + 2x^{15} + x^{14} + x^{13} + x^{11} + 2x^{10} + x^9 + 2x^8 + x^5 + 2x^4 + 2x^3 + 2x^2 + 1)
 %(x^{20} + 2x^{19} + 2x^{18} + x^{15} + x^{14} + 2x^{13} + 2x^{11} + 2x^{10} + 2x^9 + 2x^8 + 2x^5 + 2x^4 + x^3 + 2x^2 + 1)$ in $\fqn_3[x]$.
\end{example}
Now assume that $f(\beta^M)=0$ for some $\beta\in\fqn_{q^k}$, where 
$f\in\fq[x]$ is a SRIM polynomial of degree $2n$. Then minimal polynomial
of $\beta$ must divide $f(x^M)$. Hence to determine all possible $k$,
we should know about the irreducible factors of $f(x^M)$.
From \cite{bu} we know that the irreducible factors of $f(x^M)$ solely
depends on  the degree and the exponent of the irreducible
 polynomial, which is defined to be the
multiplicative order of a root of $f$ is the spliting field of $f$. Since all the 
roots are conjugate to each other, we have that the exponent is unique
data attached to the polynomial $f$. This necessiates to find the number 
of irreducible polynomial which has 
exponent $e$. We have the following 
\begin{lemma}\label{033}
Let $N_M^{*,e}(q,2n)$ denotes the number of SRIM polynomials of
degree $2n$ and exponent $e$ which are not $M^*$-power SRIM polynomial. Then we have 
\begin{align*}
N_M^{*,e}(q,2n)=
\dfrac{1}{2n}\displaystyle
\sum\limits_{\substack{l=\textup{odd}\\l|2n}}\mu(l)\phi(e)
-
\dfrac{1}{2n(M,q^{2n}-1)}\displaystyle
\sum\limits_{\substack{l=\textup{odd}\\l|2n}}\mu(l)(M(q^{2n/l}-1),e).
\end{align*}
\end{lemma}
\begin{proof}
Let us first find out number of SRIM polynomials of degree $2n$ and 
exponent $e$ in $\fq[x]$. Note that $e$ must divide $q^n+1$ as $\alpha^{q^n+1}=1$ (by Lemma \ref{001}).
Since $\fqn_{q^{2n}}^*$ is cyclic
group the number of elements of order $e$ is given by 
$\phi(e)$. But we want to have that such an element should not belong
to any proper subfield of $\fqn_{q^{2n}}$ i.e. $e$ should not divide
any $q^{\frac{n}{l}}+1$ where $l$ is odd. Since we are considering
SRIM polynomials, by inclusion-exclusion we have that number of primitive
elements in $\fqn_{q^{2n}}$ of exponent $e$ is 
$\displaystyle
\sum\limits_{\substack{l=\textup{odd}\\l|2n}}\mu(l)\phi(e)$, whence 
 number of irreducible polynomials of degree $2n$ and 
exponent $e$ in $\fq[x]$ is $\dfrac{1}{2n}\displaystyle
\sum\limits_{\substack{l=\textup{odd}\\l|2n}}\mu(l)\phi(e)$.

Next we find out the number of $M^*$-power SRIM polynomial of
degree $2n$ and exponent $e$. As in the remarks preceding \ref{007},
replacing $\alpha^{q^n+1}=1$ by the condition $\alpha^e=1$, we have that 
number of $M^*$-power SRIM polynomial of
degree $2n$ and exponent $e$ is 
$\dfrac{1}{2n(M,q^{2n}-1)}\displaystyle
\sum\limits_{\substack{l=\textup{odd}\\l|2n}}\mu(l)(M(q^{2n/l}-1),e)$.
Hence the result follows.
\end{proof}

By similar line of arguments and the fact that $x\in\fqn_{q^n}$ if and 
only if $x^{q^n-1}=1$, we have the following lemmas, which will
help us in counting. These are some generalized results of \cite{ks}, proof of which are as same as above.
\begin{lemma}\label{009}
Let $N_M(q,k)$ denotes the number of $M$-power polynomial of degree 
$k$. Then \begin{equation}
N_M(q,k)=\dfrac{1}{k(M,q^{k}-1)}\displaystyle
\sum\limits_{l|k}\mu(l)(M(q^{k/l}-1),q^k-1).
\end{equation}
\end{lemma}
\begin{lemma}\label{010}
Let $k|n$ and $N_M(q,n,\frac{n}{k})$ denotes the number of irreducible monic polynomial $f$ over $\fq$ of 
degree $\frac{n}{k}$, such that any $M$-th root of $\alpha$ (where $f(\alpha)=0$) lies in $\fqn_{q^n}$, but not in any 
proper subfield of $\fqn_{q^n}$. Then
\begin{equation}
N_M\left(q,n,\frac{n}{k}\right)=
\dfrac{1}{k}\sum\limits_{\substack{s < k,\\ \frac{n}{k}|\frac{n}{s}}}\mu(s)
\dfrac{1}{(M,q^{\frac{n}{s}}-1)}
\sum\limits_{l|\frac{2n}{k}}
\mu\left(l\right)\left(M\left(q^{\frac{n}{kls}}
-1\right),q^{\frac{n}{s}}-1\right).
\end{equation}
\end{lemma}
\begin{lemma}\label{035}
Let $N_M^{e}(q,n)$ denotes the number polynomials of
degree $n$ and exponent $e$ which are not $M$-power polynomial. Then we have 
\begin{align*}
N_M^{e}(q,n)=
\dfrac{1}{n}\displaystyle
\sum\limits_{\substack{l|n}}\mu(l)\phi(e)
-
\dfrac{1}{n(M,q^{n}-1)}\displaystyle
\sum\limits_{\substack{l|2n}}\mu(l)(M(q^{n/l-1}),e).
\end{align*}
\end{lemma}

Let $R^*_M(q,2n)$ denotes the number of pairs $\{\phi,\phi^*\}$, where 
$\phi~(\neq\phi^*)$ is an irreducible monic polynomial of degree $n\geq 2$ and $\phi$ is an $M$-power polynomial. 
Then \begin{align*}R^*_M(q,2n)=\begin{cases}
\frac{1}{2}N_M(q,n) & n\text{ is odd}\\
\frac{1}{2}\left(N_M(q,n)-N_M^*(q,n)\right)& n\text{ is even}
\end{cases}.
\end{align*}
Let $k|n$ and $R^*_M(q,2n,\frac{2n}{k})$ denotes the number of pairs $\{\phi,\phi^*\}$, where 
$\phi~(\neq\phi^*)$ is an irreducible monic polynomial $f$ over $\fq$ of 
degree $\frac{n}{k}$, such that any $M$-th root of $\alpha$ (where $f(\alpha)=0$) lies in $\fqn_{q^n}$, but not in any 
proper subfield of $\fqn_{q^n}$. Then 
 \begin{align*}R^*_M(q,2n,\frac{2n}{k})=\begin{cases}
\frac{1}{2}\left(N_M(q,n,n/k)-N_M^*(q,n,n/k)\right)& n\text{ is even}, k\text{ is odd}\\
\frac{1}{2}N_M(q,n,n/k) & \text{otherwise}\\
\end{cases}.
\end{align*}

Let $R^{*,e}_M(q,2n)$ denotes the number of pairs $\{\phi,\phi^*\}$, where 
$\phi~(\neq \phi^*)$ is an irreducible polynomial of degree $n\geq 2$, which
is not an $M$-power polynomial. Then we have 
\begin{align*}
R^{*,e}_M(q,2n)=
\begin{cases}
\frac{1}{2}N^e_M(q,n)&n\text{ is odd}\\
\frac{1}{2}(N^e_M(q,n)-\frac{1}{n(M,q^{n}-1)}\displaystyle
\sum\limits_{l|k}\mu(l)(Mq^e,q^\frac{n}{l}+1)&n\text{ is even}
\end{cases}.
\end{align*}
With the counting in hand we now move to next section, where we calculate the generating functions in the 
indeterminate $u$.
\subsection{Auxiliary results}
Before proceeding further, we note down the following lemma, which helps
in defining the indicator function (see \ref{indicatorfunction}) corresponding to
a class of irreducible polynomials having same degree and exponent.

\begin{lemma}\label{034}
Let $f_1,f_2\in\fq[x]$ be monic irreducible polynomials of degree $n$
and exponent $e$. Then $f_1(x^M)$ has a SRIM factor of degree 
$2l$ if and only 
if $f_2(x^M)$ has a SRIM factor of degree 
$2l$.
\end{lemma} 
\begin{proof}
Since $f_1$ and $f_2$ are of same degree and same exponent, by 
\cite{bu} the roots of $f_1(x^M)$ and $f_2(x^M)$
have same order. Hence the result follows from Lemma \ref{001}.
\end{proof}
\begin{lemma}
Let $f=gg^*$ be a type $2$ polynomial. Then all irreducible factors of $g(x^M)$ are of type $2$.
\end{lemma}
\begin{proof}
On the contrary if possible let $h$ be a type $1$ polynomial, which is an irreducible
polynomial of degree $2m$. Let $\alpha_1,\alpha_2,\cdots,\alpha_n,\alpha_1^{-1},\alpha_2^{-1},\cdots,\alpha_n^{-1}$ be a set of roots of $h$
in the splitting field of $h$. 
Then $\alpha_1^M,\alpha_2^M,\cdots,\alpha_n^M,\alpha_1^{-M},\alpha_2^{-M},\cdots,\alpha_n^{-M}$ are roots of $g$. Then as in Lemma \ref{004}, these elements are the only roots of $g$. Now if for all $1\leq i\leq n$, $\alpha_i^M\neq\alpha_i^{-M}$, then $g$ will be a self-reciprocal polynomial. 
Hence there exists $j$ such that $\alpha_j^M=\alpha_j^{-M}$, which implies that
$\alpha_j^M=\pm1$. This is a contradition, since $\pm 1$ are not roots of $g$.
\end{proof}
\begin{corollary}
Let $f=gg^*$ be a type $2$ polynomial of degree $2n$. Then $\alpha^M=C_f$ has a solution in $\Sp$
if and only if $g$ is an $M$-power polynomial.
\end{corollary}
\begin{proof}
Follows from the same line of proof as of Lemmas \ref{004} and \ref{005}.
\end{proof}
\begin{lemma}
Let $f$ be a type $1$ polynomial. Then all irreducible factors of $f(x^M)$ are of type $1$.
\end{lemma}
\begin{proof}
If possible, on contrary assume the $h=gg^*$ is a factor of $f(x^M)$ of type $2$. Then there exists a root $\alpha$ of $g$ such that $\alpha^{-1}$ is not a root of $g$. As in Lemma \ref{004}, if $\Lambda=\{\alpha=\alpha_1,\alpha_2,\cdots,\alpha_k\}$ are roots of $g$, we get that the only roots of $f$ are $\Lambda^M=\{\alpha^M=\alpha_1^M,\alpha_2^M,\cdots,\alpha_k^M\}$. Since $f$ is of type $1$, we have that $\pm1$ is not a root of $f$. Since the $M$-th powers might be the same, we choose a complete set of roots from $\Lambda^M$ of $f$, say $\{\alpha_1^M,\alpha_2^M,\cdots,\alpha_j^M\}$, after reindexing the set, if necessary. Note that for all $\alpha_m\in\Lambda$, we have $\alpha_m^M\neq\alpha_m^{-M}$. Since $\Lambda^M$ is closed under taking inverses, we see that there exists $\alpha_l$ such that $\alpha_1^{-M}=\alpha_l^M$ which implies that $\alpha_1^{M}=\alpha_l^{-M}$, which is a contradiction.

\end{proof}

Now we want to calculate the number of $M^*$-power SRIM
 polynomial, which contributes to finding out the generating function for
the number of separable conjugacy classes in $\Sp$.
\begin{define}
For a polynomial $f\in\fq[x]$, define \textbf{$M$-power spectrum of $f$} 
to be the set of degrees, of the irreducible factors of $f(x^M)$. Denote 
the set $M$-power spectrum of $f$ by $\mathcal{D}_M(f)$. Define
the \textbf{$M^*$-power spectrum of $f$} to be the set 
$\{l\in\mathcal{D}_M(f):f(x^M)\text{ has a SRIM factor of degree }l \}$,
which will be denoted as $\mathcal{D}_M^*(f)$.
\end{define}
\begin{remark}
We have that $f$ is an $M$-power polynomial (or $M^*$-power 
polynomial) if and only if 
$M\in\mathcal{D}_M(f)$.
\end{remark}
\begin{define}\label{indicatorfunction} For a non $M^*$-power SRIM polynomial $f$,
define the infinite product 
\begin{align*}
G_f(u)=\dfrac{1}{
\prod\limits_{i\in\mathcal{D}^*_M(f)}\left(1-u^\frac{i}{2}\right)
\prod\limits_{j\in\mathcal{D}_M(f)\setminus
\mathcal{D}_M^*(f)}\left(1-u^j\right)}.
\end{align*}
Define the \textbf{indicator function corresponding to $f$} be the function
$\mathcal{I}_M(f):\mathbb{N}\rightarrow\{0,1\}$ as follows
\begin{align*}
\mathcal{I}_M(f)(k)=\begin{cases}
1\text{ if coefficient of }u^k \text{ in }G_f(u)\neq 0\\
0\text{ otherwise}
\end{cases}.
\end{align*}
\end{define}
\begin{remark}\label{035}
Because of \ref{034} the indicator function is same for all irreducible
polynomial $f$ of degree $n$ and exponent $e$. Hence we will
denote it by $\mathcal{I}_{n,e}$.
\end{remark}
We end this section with the following notations, which will be used throughout frequently.
\begin{notation}
For a given matrix $X\in\textup{G}(m,\fq)$, we will use 
\begin{enumerate} 
\itemindent=-13pt
\item $\Delta_X$ to denote the attached combinatorial data,
\item $c_X(t)$ to denote the characteristic polynomial of $X$,
\item $m_X(t)$ to denote the minimal polynomial of $X$.
\end{enumerate}
\end{notation}

\section{Generating Functions for Separable  Matrices}
From this section onward we will be providing the generating functions for different class of elements. The route is as follows.
First we work with the conjugacy classes and then make use of orbit-stabilizer theorem to obtain the
corresponding generating functions concerning probability. We will start with the case of a matrix being separable.
\begin{lemma}\label{004}
Let $f$ be an $SRIM$ polynomial of degree $2k$, $k\geq 1$ and $\alpha^M=C_f$. Then $f(x^M)$ has an $SRIM$ factor of degree $2k$.
\end{lemma}
\begin{proof}
Let $f$ be a SRIM polynomial of degree $2d$ over $\fq$. Hence $f(x)=1+a_1x+a_2x^2+\cdots+a_{d-1}x^{d-1}+x^d(a_d+a_{d-1}x^{1}+a_{d-2}x^{2}+\cdots+a_1x^{d-1}+x^d).$ Then
%\begin{align*}
%C_f=
%\begin{pmatrix}[ccccc|ccccc]
%0&1&&&&&&&&\\
%&0&&&&&&&&\\
%&&&\ddots&&&&&&\\
%&&&0&1&&&&&\\
%&&&&0&&&&&-1\\
%\hline
%1&a_1&\cdots&a_{d-1}&a_d&0&0&\cdots&0&-a_d\\
%&&&&&1&0&\cdots&0&-a_{d-1}\\
%&&&&&&&\cdots&&\\
%&&&&&&&&0&-a_2\\
%&&&&&&&&1&-a_1
%\end{pmatrix}\in\SP(2d,\fqn_q).
%\end{align*}
considering $C_f\in\SP(2d,\fqn_{q^{2d}})$ we have that $C_f$ is conjugate to the matrix $\begin{pmatrix}[ccc|ccc]
\lambda_1&&&&&\\
&\ddots&&&&\\
&&\lambda_d&&&\\
\hline
&&&\lambda_d^{-1}&&\\
&&&&\ddots&\\
&&&&&\lambda_1^{-1}\\
\end{pmatrix}$, where $\{\lambda_i^{\pm 1}\}_{i=1}^d$ is the set of 
roots of $f$. Let $\alpha^M=C_f$ for some $\alpha\in\SP(2d,\fq)$. Since 
$\alpha$ is conjugate to the matrix $\begin{pmatrix}[ccc|ccc]
\alpha_1&&&&&\\
&\ddots&&&&\\
&&\alpha_d&&&\\
\hline
&&&\alpha_d^{-1}&&\\
&&&&\ddots&\\
&&&&&\alpha_1^{-1}\\
\end{pmatrix}$ in $\SP(2d,\fq^{2d})$, where $\{\alpha_i^{\pm 1}\}_{i=1}^d$ is the set of 
roots of $\textup{min}_{\fq}(\alpha)$, we have that
 $\alpha_i^M=\lambda_{j(i)}$. Without loss of generality, we may assume that 
$\alpha_i^{\epsilon M}=\lambda_i^\epsilon$, $\epsilon=\pm 1$. Hence $f(\alpha_i^{\pm M})=0$ for all $i$. 
Considering $H(x)=f(x^M)$, we see that $H(\alpha_i^{\pm 1})=0$ for 
all $i$, in particular $g=\textup{min}_{\fq}(\alpha)$ divides $H$. Since 
$\alpha\in\SP(2d,q)$, we have that $g$ is self reciprocal monic polynomial. 
If $g=g_1g_2$ for nontrivial factors $g_1,g_2$ of $g$, then 
$\textup{min}_{\fq}(\alpha^M)=f$ is not irreducible. Thus, we conclude 
that $g$ is an SRIM polynomial. 
\end{proof}
\begin{lemma}\label{005}
Let $f$ be an $SRIM$ polynomial of degree $2k$, $k\geq 1$
 and $f(x^M)$ has an $SRIM$ factor of degree $2k$. Then there exists 
$\alpha\in\SP(2k,q)$ such that $\alpha^M=C_f$.
\end{lemma}
\begin{proof}
We aim to show that $C_g^M$ is conjugate to $C_f$, where $g$ is a SRIM factor of degree $2k$, of $f(x^M)$. This is 
equivalent to showing that the sets $A=\{\alpha_i^{ M}: 
 i=1,2,\cdots, 2k\}$ and $\Lambda=\{\lambda_i:
i=1,2,\cdots, 2k\}$ are in bijective correspondence, where 
$\{\alpha_i\}_{i=1}^{2k}$ is the set of 
roots of $g$ and $\{\gamma_i\}_{i=1}^{2k}$ is the set of 
roots of $f$. Since $f$ is separable, we have that $|\Lambda|=2k$.

Note that in $\fqn_{q^{2k}}$, we have 
$f(x)=\displaystyle\prod_{i=1}^{2k}(x-\lambda_i)$, 
$g(x)=\displaystyle\prod_{i=1}^{2k}(x-\alpha_i)$. Since $g(x)$ divides 
$f(x^M)$, we have that, for all $j$,
$0=f(\alpha_j^M)=\displaystyle\prod_{i=1}^{2k}(\alpha_j^M-
\lambda_i)$. Hence $\alpha_1^M=\lambda_i$ for 
some $i$. After some permutation, we may assume that $i=1$.
Note that if  $h$ is the characteristic polynomial of $C_g^M$, then $h(\alpha_1)=0$. 
Since minimal polynomial of $\alpha_1$ is $f$, we have that $f=h$. 
Since $f$ is separable, we have that $|A|=|\Lambda|=2k$. 
\end{proof}

\begin{corollary}\label{006}
Let $A\in\Sp$ has characteristic polynomial $f$, which is SRIM of degree
$2n$. Then $\alpha^M=A$, has a solution in $\Sp$, if and only if $f$ is  
$M^*$-power SRIM polynomial.
\end{corollary}
\begin{proposition}\label{011}
Let $cs^M_{\SP}(n,q)$ denotes the number of $M$-power separable
 conjugacy classes in $\Sp$ and 
$cS^M_{\SP}(q,u)=1+\sum\limits_{m=1}^{\infty}cs^M_{\SP}(m,q)u^m$. Then 
\begin{equation}
cS^M_{\SP}(q,u)=\displaystyle\prod_{d=1}^{\infty}\left(1+u^d\right)^{N_M^*(q,2d)}\prod_{d=1}^{\infty}\left(1+u^d\right)^{R_M^*(q,2d)}.
\end{equation}
\end{proposition}
\begin{proof}
Let $X\in\Sp$ be a separable matrix. Then $c_X(t)$ is separable and is a
product of $*$-irreducible polynomials. Since $c_X(t)$ is separable
we have that each of the factor in $c_X(t)$ occurs exactly once. Considering
the fact that $X$ has determinant $1$, any $*$-irreducible
polynomial of type $3$ must occur twice. Hence none of the polynomial
$t\pm 1$ is a factor of $c_X(t)$. Let $\Delta_X=\{(f,\lambda_f):f\in\Phi\}$. Then $\Delta_X$ represents a separable class if and only if
\begin{enumerate}
\itemindent=-13pt
\item $\lambda_{t\pm 1}=0$,
\item $\lambda_f=\lambda_{f^*}\in\{0,1\}$,
\item $\sum\limits_{f|c_X}\deg f=2n$.
\end{enumerate}
Hence using Corollary \ref{006}, we have that $X$ is an $M$-th power separable 
element if and only if
\begin{enumerate}
\itemindent=-13pt
\item for all $(f,1)\in\Delta_X$ and $f=f^*$, $f\in\Phi^*_M$,
\item for all $(f,1)\in\Delta_X$ and $f\neq f^*$, $f\in\Phi_M$.
\end{enumerate}
Thus 
$c_X(t)=\prod\limits_{i=1}^rf_i\prod\limits_{j=1}^sg_jg_j^*$, where
$f_i$ is an $M^*$-power SRIM polynomial and $g_j\neq g_j^*$ is an 
$M$-power polynomial. Considering the fact that each of the factors 
$f_i$ and $g_jg^*_j$ is of even degree, we have that
\begin{align*}
cS^M_{\SP}(q,u)
&=\displaystyle\prod_{f\in\Phi_M^*}\left(1+u^{\frac{\deg f}{2}}\right)
\prod_{g\in\Phi_M\setminus\Phi_M^*}\left(1+u^{ \deg g}\right)^\frac{1}{2}\\
&=\displaystyle\prod_{d=1}^{\infty}\left(1+u^d\right)^{N_M^*(q,2d)}\prod_{d=1}^{\infty}\left(1+u^d\right)^{R_M^*(q,2d)}.
\end{align*}
\end{proof}
\begin{theorem}\label{012}
Let $s^M_{\SP}(n,q)$ denotes the probability of an element to be
 $M$-power separable in $\Sp$ and 
$S^M_{\SP}(q,u)=1+\sum\limits_{m=1}^{\infty}s^M_{\SP}(m,q)u^m$. Then 
\begin{equation}
S^M_{\SP}(q,u)=\displaystyle\prod_{d=1}^{\infty}\left(1+\dfrac{u^d}{q^d+1}\right)^{N_M^*(q,2d)}\prod_{d=1}^{\infty}\left(1+\dfrac{u^d}{q^d-1}\right)^{R_M^*(q,2d)}.
\end{equation}
\end{theorem}
\begin{proof}
From Lemmas \ref{002} and \ref{003}, it follows that
\begin{enumerate}
\itemindent=-13pt
\item for $X\in\Sp$, if $c_X(t)$ is a SRIM polynomial then the centraliser
of $X$ inside $\Sp$ is of order $q^n+1$,
\item for $X\in\Sp$, if $c_X(t)$ is $*$-irreducible polynomial of type $2$
 then the centraliser
of $X$ inside $\Sp$ is of order $q^n-1$.
\end{enumerate}
Hence using Proposition \ref{011} and the fact that centraliser of a general block diagonal 
matrix is a direct sum of each of the corresponding centralisers, we have 
\begin{equation*}
S^M_{\SP}(q,u)=\displaystyle\prod_{d=1}^{\infty}\left(1+\dfrac{u^d}{q^d+1}\right)^{N_M^*(q,2d)}\prod_{d=1}^{\infty}\left(1+\dfrac{u^d}{q^d-1}\right)^{R_M^*(q,2d)}.
\end{equation*}
\end{proof}
The next theorem is proved along the same line of proof of Theorem $2.3.1$ of \cite{fnp}.
\begin{theorem}\label{025}
Let $s_{\On^\e}^M(n,q)$ denotes the probability of an element to be 
$M$-power separable in $\On^\e(2n,q)$ with $\e\in\{\pm\}$ and
$s_{\On^0}^M(n,q)$ denotes the probability of an element to be 
$M$-power separable in $\On^0(2n+1,q)$. Define 
\begin{align*}
S_{\On^+}^M(q,u)&=1+\sum\limits_{m\geq 1}^\infty
s_{\On^+}^M(m,q)u^m \\  
S_{\On^-}^M(q,u)&=\sum\limits_{m\geq 1}^\infty
s_{\On^-}^M(m,q)u^m \\
S_{\On^0}^M(q,u)&=1+\sum\limits_{m\geq 1}^\infty
s_{\On^0}^M(m,q)u^m.
\end{align*} 
Then
\begin{align}
&S_{\On^+}^M(u^2)+S_{\On^-}^M(u^2)
+2uS_{\On^0}^M(u^2)=(1+u)^{o(M,q)}S_{\SP}^M(u^2),\\
& S_{\On^+}^M(u^2)-S_{\On^-}^M(u^2)=X^M_{\On^0}(u^2),
\end{align}
where
\begin{align*}
X^M_{\On^0}(q,u)=\displaystyle\prod_{d=1}^{\infty}\left
(1-\dfrac{u^d}{q^d+1}\right)^{N_M^*(q,2d)}\prod_{d=1}^{\infty}\left(1+\dfrac{u^d}{q^d-1}\right)^{R_M^*(q,2d)},
\end{align*}
where 
\begin{align*}
o(M,q)=\begin{cases}
1 & \text{if } M \text{ even }\\
2 &\text{otherwise}
\end{cases}.
\end{align*}
\end{theorem}
\begin{proof}
The proof is similar to that of the previous theorem. But if $X$ is a separable orthogonal
matrix, then $t\pm 1$ can divide $c_X(t)$. The multiplicity of $t\pm 1$ 
in $c_X(t)$ can be at most $1$, because $c_X(t)$ is separable.
Since center of $\On^\e(m,q)$ is $\{\pm\mathds{1}\}$, we have that
the block corresponding to $t+1$, of size $1\times 1$ is an $M$-th power
if and only if $M$ is odd.
Now suppose $M$ is even. Consider the product 
\begin{align*}
(1+u)\displaystyle\prod_{d=1}^{\infty}\left
(1+\dfrac{u^d}{q^d+1}\right)^{N_M^*(q,2d)}\prod_{d=1}^{\infty}\left(1+\dfrac{u^d}{q^d-1}\right)^{R_M^*(q,2d)}.
\end{align*}
For the case $M$ being even,
write $(1+u)$ as $(1+\frac{u}{2}+\frac{u}{2})$ and this
tracks the possibility of 
$t-1$ dividing $c_X(t)$. Each term $\frac{u}{2}$, appears for the distinct 
conjugacy classes corresponding to $t-1$, each having order of centraliser
$2$. Note that in this case $-1$ is not an $M$-th power. Hence 
$e(M,q)=1$. Now for $n$ even positive, the coefficient of $u^n$ is
 $s_{\On^+}^M(n,q)+s_{\On^-}^M(n,q)$, where as for $n$ being 
odd positive the coefficient is $e(q)s_{\On^0}^M(n,q)$ for $e(q)$ many 
types of forms over $\fq^n$. 

For the case $M$ being odd and $q$ being odd, consider the product
\begin{align*}
(1+u)^2\displaystyle\prod_{d=1}^{\infty}\left
(1+\dfrac{u^d}{q^d+1}\right)^{N_M^*(q,2d)}\prod_{d=1}^{\infty}\left(1+\dfrac{u^d}{q^d-1}\right)^{R_M^*(q,2d)}.
\end{align*}
Then writing $(1+u)$ as $(1+\frac{u}{2}+\frac{u}{2})$ we get the
possibility of $t-1$ dividing $c_X(t)$. But there are two such conjugacy
classes each having centraliser of size $2$. The same argument applies 
for the polynomial $t+1$ as well. Hence the power $2$. This proves the first equation.

We will prove the second equation by modifying first equation. For each 
of the factor $1+C_f u^{2d}$ in the right hand side of the first equation,
 where $C_f$ is the reciprocal of the size
of the corresponding centraliser, replace it by $1+\tau_f C_f 
u^{2d} $, where $\tau_f=\tau(V_f)$ and $V_f$ denotes the
component of $V$ corresponding to $f$, in the primary 
decomposition as an $\fq[X]$ module. Then since for $q$ odd, each of 
the term $\frac{u}{2}$ corresponds to conjugacy class with 
$\tau_f$ values being negative to each other, the term $(1+u)$ vanishes. 
Now, since $\tau_f=-1$, when $f$ is of type $1$ and 
$\tau_f=+1$, when $f$ is of type $2$, the factors $1+\dfrac{u^{2d}}{q^d+1}$
are replaced by $1-\dfrac{u^{2d}}{q^d+1}$, whereas the factor
$1+\dfrac{u^{2d}}{q^d-1}$ remains as it is. Hence the product 
becomes $\displaystyle\prod_{d=1}^{\infty}\left
(1-\dfrac{u^d}{q^d+1}\right)^{N_M^*(q,2d)}\prod_{d=1}^{\infty}\left(1+\dfrac{u^d}{q^d-1}\right)^{R_M^*(q,2d)}$, which on expanding gives
$S_{\On^+}^M(u^2)-S_{\On^-}^M(u^2)$.
\end{proof}
\begin{remark}\label{027}
Let $cs_{\On^\e}^M(n,q)$ denotes the probability of a conjugacy class
 to be 
$M$-power separable in $\On^\e(2n,q)$ with $\e\in\{\pm\}$ and
$cs_{\On^0}^M(n,q)$ denotes the probability of a conjugacy class
 to be 
$M$-power separable in $\On^0(2n+1,q)$. Define 
\begin{align*}
cS_{\On^+}^M(q,u)&=1+\sum\limits_{m\geq 1}^\infty
cs_{\On^+}^M(m,q)u^m, \\  
cS_{\On^-}^M(q,u)&=\sum\limits_{m\geq 1}^\infty
cs_{\On^-}^M(m,q)u^m, \\
cS_{\On^0}^M(q,u)&=1+\sum\limits_{m\geq 1}^\infty
cs_{\On^0}^M(m,q)u^m.
\end{align*} 
Then
\begin{align}
&cS_{\On^+}^M(u^2)+cS_{\On^-}^M(u^2)
+2ucS_{\On^0}^M(u^2)=(1+2u)^{o(M,q)}cS_{\SP}^M(u^2),\\
& cS_{\On^+}^M(u^2)-cS_{\On^-}^M(u^2)=cX^M_{\On^0}(u^2),
\end{align}
where
\begin{align*}
cX^M_{\On^0}(q,u)=\displaystyle\prod_{d=1}^{\infty}\left
(1-u^d\right)^{N_M^*(q,2d)}
\prod_{d=1}^{\infty}\left(1+u^d\right)^{R_M^*(q,2d)}.
\end{align*}

\end{remark}

\section{Generating Functions for Semisimple  Matrices}

Before moving towards the determination of generating functions for
semisimple case, we find out the cases where $M$-th root of 
$-\mathds{1}$ exists. This is certainly true, whenever $M$ is odd. The next lemma discusses the scenario, when $M$ is even. 

We first note that for a number $M=2^nM'$ for $(M',2)=1$, we have
\begin{align*}
x^M+1=(x^{2^n}+1)(x^{2^n(M'-1)}-x^{2^n(M'-2)}+\cdots+1).
\end{align*}
Let us call the second factor in the above decomposition $F$. If we can prove that all irreducible factors of $F$ have degree which is multiple of degree of the factors of $x^{2^n}+1$, we can conclude about factors of $x^M+1$. We note the following result from \cite{m}.
\begin{proposition}[Theorem $1$, \cite{m}]
Let $q\equiv3\pmod{4}$, i.e., $q= 2^Am - 1$, $A  \geq2$, $m$ odd. Let $n\geq2$.
\begin{enumerate} 
\itemindent=-13pt
\item If $n < A$, then $x^{2^n}+1$ is the product of $2n-1$ irreducible trinomials
over $\fq$
\begin{align*}
x^{2^n}+1=\prod\limits_{\gamma\in\Gamma}(x^2+\gamma x+1),
\end{align*}
where $\Gamma$ is the set of all roots of the Dickson polynomial $D_{2^{n-1}} (x, 1)$.
\item If $n \geq A$, then $x^{2^n}+1$ is the product of $2^{A-1}$ irreducible polynomials
over $\fq$
\begin{align*}
x^{2^n}+1=\prod\limits_{\delta\in\Delta}(x^{2^{n-A+1}}+\delta x^{2^{n-A}}-1),
\end{align*}
where $\Delta$ is the set of all roots of the Dickson polynomial $D_{2^{A-1}} (x,-1)$.
\end{enumerate}
\end{proposition}
Suppose $M=2^n$, for some $n$. Then the block $\begin{bmatrix}
-1&0\\0&-1
\end{bmatrix}$ is an $M$-th root in the first case. For the second case the $L\times L$ matrix $-\mathds{1}_L$ is an $M$-th power if and only if $L$ is a multiple of $2^{n-A+2}$.
\begin{proposition}\label{014}
Let $css^M_{\SP}(2n,q)$ denotes the number of $M$-power semisimple
 conjugacy classes in $\Sp$ and 
$cSS^M_{\SP}(q,u)=1+\sum\limits_{m=1}^{\infty}css^M_{\SP}(2m,q)u^{m}$. Then $cSS^M_{\SP}(q,u)$ is given by 
\begin{align*}
&\dfrac{1}{(1-u^{r(M,q)})(1-u)}
\displaystyle
\prod_{d=1}^{\infty}(1-u^d)^{-N_M^*(q,2d)}
\prod_{d=1}^{\infty}(1-u^d)^{-R_M^*(q,2d)}\\
&\times\prod\limits_{d=1}^\infty
\prod\limits_{e|q^d+1}\left(1+\sum\limits_{m=1}^\infty\mathcal{I}_{e,2d}(2dm)u^{dm}\right)^{N_M^{*,e}(q,2d)}
\prod\limits_{d=1}^\infty
\prod\limits_{e|q^d-1}\left(1+\sum\limits_{m=1}^\infty\mathcal{I}_{e,d}(dm)u^{dm}\right)^{R_M^{*,e}(q,2d)},
\end{align*}
where 
$r(M,q)=\begin{cases}
1&\text{if }M\text{ is odd}\\
2&q\equiv3\pmod{4},n<A\\
2^{n-A+1}&q\equiv3\pmod{4},n\geq A
\end{cases}.$
\end{proposition}
\begin{proof}
Let $X\in\Sp$ be semisimple. Then $m_X(t)$ is a product of distinct $*$-irreducible
polynomials. Considering that $X$ has determinant $1$, we have that $(t+1)$ has even multiplicity in $c_X(t)$. 
This forces to have multiplicity of $t-1$ to be even in $c_X(t)$. Let $\Delta_X=\{(f,\lambda_f):f\in\Phi\}$. 
Then $X$ is semisimple if and only if
\begin{enumerate}
\itemindent=-13pt
\item $\lambda_{t+1}\in\{0,1^{2r_1}\}$, 
$\lambda_{t-1}\in\{0,1^{2r_{-1}}\}$, 
\item $\lambda_f=\lambda_{f^*}\in\{0,1^{l_f}\}$,
\item $\sum|\lambda_f|=2n$,
\end{enumerate}
where $r_1,r_{-1},l_f\in\mathbb{Z}_{>0}$. Hence using Corollary \ref{006} and 
discussion preceding Lemma \ref{033}, $X$ is an $M$-th power if and only if
\begin{enumerate}
\itemindent=-13pt
\item $\lambda_{t-1}\in\{0,1^{2r_1}\}$, $r_1\in\mathbb{Z}_{>0}$, 
\item $\lambda_{t+1}\in\{0,1^{2r_{-1}}\}$, where $r_{-1}\in r(M,q)\mathbb{Z}_{>0}$,
\item for $f$, an $M*$-power SRIM polynomial of degree $d$ we have 
$\lambda_f\in\{0,1^m:m\in\mathbb{Z}_{>0}\}$,
\item for $f~(\neq f^*)$, an $M$-power polynomial of degree $d$ we have
\begin{align*}
\lambda_f=\lambda_{f^*}\in\{0,1^m:m\in\mathbb{Z}_{>0}\}    
\end{align*}
\item for $f$, a type $1$ polynomial which is not an $M^*$-power polynomial,
\begin{align*}
 \lambda_f\in\{0,1^m:m\in\sum\limits_{i\in\mathcal{D}_M(f)}\mathbb{Z}_{\geq 0}i\},   
\end{align*}
\item for $f$, a type $2$ polynomial which is not an $M$-power polynomial\begin{align*}
 \lambda_f=\lambda_{f^*}\in\{0,1^m:m\in\sum\limits_{i\in\mathcal{D}_M(f)}\mathbb{Z}_{\geq 0}i\}.   
\end{align*}
\end{enumerate}
 Hence $cSS^M_{\SP}(q,u)$ is given by
\begin{align*}
&\left(1+\sum\limits_{m=1}^\infty u^m\right)
\left(1+\sum\limits_{m=1}^\infty u^{m\frac{r(M,q)}{2}}\right)^{e(q)-1}
\\
\times&\prod\limits_{\substack{f=f^*\\f\in\Phi^*_M}}
\left(1+\sum\limits_{m=1}^\infty u^{m\frac{\deg f}{2}}
\right)
\prod\limits_{
\substack{\{f\neq f^*\}\\f\in\Phi_M}
}
\left(1+\sum\limits_{m=1}^\infty u^{m\deg f}
\right)
\\
\times&\prod\limits_{\substack{f=f^*\\f\not\in\Phi^*_M}}
\prod\limits_{e|q^{\frac{\deg f}{2}}+1}\left(1+\sum\limits_{m=1}^\infty\mathcal{I}_{M}(f)(m\deg f)u^{\frac{\deg f}{2}m}\right)\\
\times&\prod\limits_{\substack{\{f,f^*\}\\f\neq f^*,f\not\in\Phi_M}}
\prod\limits_{e|q^{\deg f}-1}\left(1+\sum\limits_{m=1}^\infty\mathcal{I}_{M}(f)(m\deg f)u^{m\deg f}\right)
\end{align*}
where \begin{enumerate}
\itemindent=-13pt
\item the first term accounts for the polynomial $t-1$,
\item the second term accounts for the polynomial $t+1$, which vanishes when $(q,2)\neq 1$ and hence the power $e(q)-1$,
\item the third and fourth term appear for the polynomials in $\Phi^*_M$ and $\Phi_M$ respectively,
\item the fifth term appears for the type $1$ polynomial which are not 
$M^*$-th power SRIM. Note that in this case $f(x^M)$ has factors of degrees
belonging to $\mathcal{D}_M(f)$. Suppose $k_i\in\mathcal{D}_M(f)$ and
$g_{k_i}$ be a factor of $f(x^M)$, of degree $k_i$ with
$i=1,2,\cdots$. Then clearly $\deg f|k_i$ and $(f,1^{\frac{k_i}{\deg f}})$ is an $M$-th power for all $k_i\in\mathcal{D}_M(f)$. Then for 
any integer $m\in\sum_i\mathbb{Z}_{\geq 0}\frac{k_i}{\deg f}$, the 
class $(f,1^m)$ is an $M$-th power.

In this case two kinds of polynomials can occur in factorization of $f(x^M)$. It can be of either type $1$ or type $2$. For this the function
$G_f$ has two components corresponding to each type. 
Hence we associate 
the function $\mathcal{I}_M(f)$ which indicates if $m\in\sum_i\mathbb{Z}_{\geq 0}\frac{k_i}{\deg f}$ or not.
\item the sixth term appears for the type $2$ polynomial which are not 
$M$-th power (applying similar kind of argument as in the previous case).
\end{enumerate}
Hence plugging in the formulae for number of each kind of polynomials
and taking into consideration \ref{034} we get the result.
\end{proof}
\begin{theorem}\label{016}
Let $ss^M_{\SP}(n,q)$ denotes the probability of an element to be
 $M$-power semisimple in $\Sp$ and 
$SS^M_{\SP}(q,u)=1+\sum\limits_{m=1}^{\infty}ss^M_{\SP}(2m,q)u^{m}$. Then 
\begin{align*}SS^M_{\SP}(q,u)&=\left(1+\sum\limits_{m\geq 1}\frac{{u}^{m\frac{r(M,q)}{2}}}{|\SP(mr(M,q),\fq)|}
\right)
\left(1+\sum\limits_{m\geq 1}\frac{{u}^{m}}{|\SP(2m,\fq)|}
\right)
\displaystyle\\
\times&\prod_{d=1}^{\infty}
\left(1+\sum\limits_{m\geq 1}
\dfrac{u^{dm}}{|\U(m,\fqn_{q^d})|}\right)^{N_M^*(q,2d)}
\prod_{d=1}^{\infty}
\left(1+\sum\limits_{m\geq 1}
\dfrac{u^{dm}}{|\GL(m,\fqn_{q^d})|}\right)^{R_M^*(q,2d)}\\
\times&\prod\limits_{d=1}^\infty
\prod\limits_{e|q^d+1}\left(1+\sum\limits_{m=1}^\infty\mathcal{I}_{e,2d}(2dm)\dfrac{u^{dm}}{|\U(m,\fqn_{q^d})|}\right)^{N_M^{*,e}(q,2d)}\\
\times&\prod\limits_{d=1}^\infty
\prod\limits_{e|q^d-1}\left(1+\sum\limits_{m=1}^\infty\mathcal{I}_{e,d}(dm)\dfrac{u^{dm}}{|\GL(m,\fqn_{q^d})|}\right)^{R_M^{*,e}(q,2d)}
\end{align*}
where $r(M,q)$ is the function defined in Proposition \ref{014}.
\end{theorem}
\begin{proof}
Since $\pm\mathds{1}_{2n}$ are in the center of $\Sp$, their centralisers are 
$\Sp$ itself. From Lemmas \ref{002} and \ref{003}, we have that
\begin{enumerate}
\itemindent=-13pt
\item if $f\in\Phi^*$ is of
degree $2k$ and $\mu_f=1^{m}$, then the
centraliser of $X$ inside $\SP(2km,\fq)$ is of order $|\U(m,q^{k})|$ ,
\item if $f\in\Phi\setminus
\Phi^*$ is of degree $k$ and $\mu_f=1^{m}$,
 then the
centraliser of $X$ inside $\SP(2dm,\fq)$ is of order $|\GL(m,q^d)|$.
\end{enumerate}
Hence using Proposition \ref{014} and the fact that centraliser of a general block diagonal 
matrix is direct sum of each of the corresponding centralisers, we have 
the result.
\end{proof}
We define the following functions for simplifying the statements in the case of 
orthogonal groups. These are motivated by the definitions in Chapter $3$ of \cite{fnp}.
\begin{align*}
Y_1^{*,M}(u,q)=&\prod_{d=1}^{\infty}
\left(1+\sum\limits_{m\geq 1}
\dfrac{u^{dm}}{|\U(m,q^d)|}\right)^{N_M^*(q,2d)}
\prod_{d=1}^{\infty}
\left(1+\sum\limits_{m\geq 1}
\dfrac{u^{dm}}{|\GL(m,q^d)|}\right)^{R_M^*(q,2d)}\\
\times&\prod\limits_{d=1}^\infty
\prod\limits_{e|q^d+1}\left(1+\sum\limits_{m=1}^\infty\mathcal{I}_{e,2d}(2dm)\dfrac{u^{dm}}{|\U(m,q^d)|}\right)^{N_M^{*,e}(q,2d)}\\
\times&\prod\limits_{d=1}^\infty
\prod\limits_{e|q^d-1}\left(1+\sum\limits_{m=1}^\infty\mathcal{I}_{e,d}(dm)\dfrac{u^{dm}}{|\GL(m,q^d)|}\right)^{R_M^{*,e}(q,2d)},
\end{align*}
\begin{align*}
Y_2^{*,M}(u,q)=&\prod_{d=1}^{\infty}
\left(1+\sum\limits_{m\geq 1}
\dfrac{(-1)^mu^{dm}}{|\U(m,q^d)|}\right)^{N_M^*(q,2d)}
\prod_{d=1}^{\infty}
\left(1+\sum\limits_{m\geq 1}
\dfrac{u^{dm}}{|\GL(m,q^d)|}\right)^{R_M^*(q,2d)}\\
\times&\prod\limits_{d=1}^\infty
\prod\limits_{e|q^d+1}\left(1+\sum\limits_{m=1}^\infty\mathcal{I}_{e,2d}(2dm)\dfrac{(-1)^mu^{dm}}{|\U(m,q^d)|}\right)^{N_M^{*,e}(q,2d)}\\
\times&\prod\limits_{d=1}^\infty
\prod\limits_{e|q^d-1}\left(1+\sum\limits_{m=1}^\infty\mathcal{I}_{e,d}(dm)\dfrac{u^{dm}}{|\GL(m,q^d)|}\right)^{R_M^{*,e}(q,2d)},
\end{align*}\begin{align*}
F_{+,+1}(u,q)=&1+\sum\limits_{m\geq 1}\left(
\dfrac{1}{|\On^+(2m,q)|}+\dfrac{1}{|\On^-(2m,q)|}
\right)u^m,\\
F_{-,+1}(u,q)=&1+\sum\limits_{m\geq 1}\left(
\dfrac{1}{|\On^+(2m,q)|}-\dfrac{1}{|\On^-(2m,q)|}
\right)u^m,\\
F_{+1}(u,q)=&1+\sum\limits_{m\geq 1}\dfrac{u^m}{|\SP(2m,q)|}\\
F_{+,-1}^M(u,q)=&
1+\sum\limits_{m\geq 1}
\left(
\dfrac{1}{|\On^+(mr(M,q),q)|}+
\dfrac{1}{|\On^-(mr(M,q),q)|}
\right)
{u}^{m\frac{r(M,q)}{2}},
\end{align*}\begin{align*}
F_{-,-1}^M(u,q)=&
1+\sum\limits_{m\geq 1}
\left(
\dfrac{1}{|\On^+(mr(M,q),q)|}-
\dfrac{1}{|\On^-(mr(M,q),q)|}
\right)
{u}^{m\frac{r(M,q)}{2}},\\
F_{-1}(u,q)=&1+\sum\limits_{m\geq 1}\dfrac{u^\frac{mr(M,q)}{2}}{|\SP(mr(M,q),q)|}.
\end{align*}
The next theorem is proved along the same as Theorem 
$3.1.6$ of \cite{fnp}.
\begin{theorem}\label{028}
Let $ss_{\On^\e}^M(n,q)$ denotes the probability of an element to be 
$M$-power semisimple in $\On^\e(2n,q)$ with $\e\in\{\pm\}$ and
$s_{\On^0}^M(n,q)$ denotes the probability of an element to be 
$M$-power semisimple in $\On^0(2n+1,q)$. 

Define 
\begin{align*}
SS_{\On^+}^M(q,u)&=1+\sum\limits_{m\geq 1}^\infty
ss_{\On^+}^M(m,q)u^m \\  
SS_{\On^-}^M(q,u)&=\sum\limits_{m\geq 1}^\infty
ss_{\On^-}^M(m,q)u^m \\
SS_{\On^0}^M(q,u)&=1+\sum\limits_{m\geq 1}^\infty
ss_{\On^0}^M(m,q)u^m. 
\end{align*}  Then
\begin{align*}
SS_{\On^+}^M(u^2)+SS_{\On^-}^M(u^2)
+2uSS_{\On^0}^M(u^2)
=&\left(F_{+,+1}^M(u^2)+uF_{+1}(u^2)\right)\\
\times&\left(F_{+,-1}^M(u^2)+uF_{-1}(u^2)\right)^{e(q)-1}Y_1^{*,M}(u^2),\\
 S_{\On^+}^M(u^2)-S_{\On^-}^M(u^2)=&F_{-,+1}(u^2)[F_{-,-1}^M(u^2)]^{e(q)-1}Y_2^{*,M}(u^2).
\end{align*}
\end{theorem}
\begin{proof}
Using similar argument as in Theorem \ref{025}, the proof follows.
\end{proof}

\section{Generating Functions for Cyclic  Matrices}

Before we give the generating function for the cyclic conjugacy classes 
we find out which matrices 
with eigenvalue $ 1$ (or $-1$) are $M$-th power. Note that two 
matrices $A$ and $B$ are conjugate if and only if $-A$ and $-B$ are conjugate.
Hence, conjugacy classes of matrices with all eigenvalue $1$ is in bijection with
conjugacy classes of matrices with all eigenvalue $-1$.
Recall that an element $X\in\Sp$ is cyclic if and only if $c_X(t)=m_X(t)$.
 Hence we concentrate on single Jordan block with 
eigenvalue $1$.
Since Jordan blocks of odd size should occur even times, they do not contribute to cyclic elements (refer \cite{f2}, pp 48). 
Although, we have the following 
\begin{lemma}\label{018}
If $(M,q)=1$, then every unipotent conjugacy class is an $M$-th power.
\end{lemma}
Let us denote by $U_{f,n}=\begin{pmatrix}
C_f & \mathds{1} &  & &&\\
& C_f & \mathds{1} &&&\\
& & C_f &\mathds{1}&&\\
&&&\ddots&\ddots&\\
&&&&C_f&\mathds{1}\\
&&&&& C_f
\end{pmatrix}$, also sometimes by $U_{C_f,n}$ the matrix of size $n\deg f\times n\deg f$, where $f$ is a monic irreducible polynomial
and $C_f$ is the standard companion matrix of $f$. 
Now we want to find the structure of semisimple
part $\alpha_s$ of $\alpha$, where 
$\alpha^M=U_{t+1,n}$ and
$\alpha\in\GL(n,q)$. We have the following 
\begin{lemma}\label{019} 
Let $(M,q)=1$ and
there exists $\alpha\in\GL(n,q)$ satisfying $\alpha^M=U_{t+1,n}$. Then
$\alpha_s$ is a scalar matrix. 
\end{lemma} 
\begin{proof}
Let $\alpha$ be conjugate to $U_{C_f,n}$ for some monic irreducible polynomial $f$. Then since $(M,q)=1$, we have that 
$\alpha^M$ is conjugate to $U_{C_f^M,n}$. 
Now $\alpha^M=U_{t+1,n}$ implies that $C_f^M=-1$, whence  
$U_{C_f,n}$ is conjugate to $\gamma\mathds{1}U_{t-1,n}$, where 
$C_f$ is denoted as $\gamma$ (as it is a $1\times 1$ matrix). So
$\alpha_s=\gamma\mathds{1}$, is a scalar matrix, as claimed.
%Let $-1=\gamma^M$ for some $\gamma\in\fq$. Then considering 
%$\alpha'=\begin{pmatrix}
%\gamma & 1 & \cdots & \\
% & \gamma & \cdots &\\
%& & \ddots &\\
%& & & \gamma
%\end{pmatrix}$, we have that $\alpha'^M$ satisfies $(\alpha'^M+\mathds{1})^{n}=0$
%and $(\alpha'^M+\mathds{1})^m\neq 0$ for all $m<n$. Hence $\alpha'^M$ and $U_{t+1,n}$ are conjugate. So $\alpha$, $M$-th root of
%$U_{t+1,n}$ exists.
%
%Conversely suppose $\alpha^M=U_{t+1,n}$. Let 
%$\alpha=\begin{pmatrix}
%C_f & \mathds{1} &  & &\\
%& C_f & \mathds{1} &&\\
%& & C_f &&\\
%&&&\ddots&\\
%&&&& C_f
%\end{pmatrix}$, (to be denoted as $U_{C_f,n}$) where $f$ is an irreducible monic polynomial 
%and $C_f$ denotes the standard companion matrix of $f$ in $\GL(n,\fq)$. 
%Note that it is enough to consider only single polynomial, as Jordan block 
%decomposition of a matrix, is direct sum of Jordan blocks corresponding to
%different irreducible monic factors of $c_\alpha(t)$. Then since 
%$(M,q)=1$, we have that $\alpha^M$ is conjugate to $U_{C_f^M}$.
% $\alpha'=\begin{pmatrix}
%C_f^M & \mathds{1} &  & &\\
%& C_f^M & \mathds{1} &&\\
%& & C_f^M &&\\
%&&&\ddots&\\
%&&&& C_f^M
%\end{pmatrix}$. 
%But $\alpha^M=U_{t+1,n}$ implies that $\alpha'$
%is conjugate to $U_{t+1,n}$ and hence $C_f$ is a $1\times 1$ matrix, whence
%$-1\in\im\theta_M$
\end{proof}
%\begin{remark}
%Any $M$-th root of $U_{t+1,n}$ is conjugate to $\gamma U_1$ for some 
%$\gamma$ satisfying $\gamma^M=-1$. 
%\end{remark}
\begin{corollary}\label{019} 
Let $(M,q)=1$ and $U_{t+1,n}^{\SP}\in\SP(2m,q)$ where $U_{t+1,n}^{\SP}$ is conjugate 
to $U_{t+1,n}$ in $\GL(2m,q)$. Then $U_{t+1,n}^{\SP}$ is an $M$-th 
power if and only if $M$ is odd.
\end{corollary}
\begin{proof}
Let $M$ be odd. Note that $-U_{t+1,n}^{\SP}$ is a unipotent element and
hence has an $M$-th root, say $\alpha$. Then $(-\alpha)^M=U_{t+1,n}^{\SP}$.

Conversely suppose 
$\alpha^M=U_{t+1,n}^{\SP}$ for some $\alpha\in\SP(2m,q)$.
 Also $U_{t+1,n}^{\SP},\alpha\in\GL(2m,q)$ implies that $\alpha_s$ is a 
scalar matrix. Then
we have that $\alpha_s=-\mathds{1}$, since $\SP(2m,q)$ contains only the 
scalar matrices $\{\pm\mathds{1}\}$. Hence $M$ should be odd.
\end{proof}
\begin{corollary}\label{022}
Let $(M,q)=1$.
Any matrix $X\in\Sp$ with combinatorial data $(t+1,m^\pm)$ is an $M$-th power
if and only if $M$ is odd.
\end{corollary}
\begin{proof}
This follows from \ref{018} and proof of \ref{019}.
\end{proof}
From Chapter 3 of \cite{bg}, we use tensor product construction to study the $\pm1$-potent conjugacy
classes and denote them by $[J_a^{b,\e}]$. From Lemma $3.4.7$ \cite{bg}, we have
\begin{lemma}
A unipotent element of type $[J_a^{b,\e}]$ fixes a pair of complementary maximal totally 
isotropic subspaces of the natural $\SP_{ab}(q)$-module if and only if $\e=+$.
\end{lemma}
\begin{corollary}
We have that $[J_a^{b,\e}]^M=[J_a^{b,\e}]$.
\end{corollary}
\begin{proof}
Let $\e=+$. Then there are complementary maximally totally isotropic subspaces $W_1,W_2$ of dimension $\frac{b}{2}$
such that $J_a\otimes I_b$ fixes $U\otimes W_1$ and $U\otimes W_2$. Then $J_a^M\otimes I_b^M$ fixes $U\otimes W_1$ and $U\otimes W_2$ and hence the 
result follows in this case.

For $\e=-$, on the contrary assume $[J_a^{b,-}]^M$ has the property that it fixes a pair of complementary maximally totally 
isotropic subspaces. Since $[J_a^{b,-}]$
has power coprime to $M$, we have that $[J_a^{b,-}]$ fixes a pair of complementary maximally totally 
isotropic subspaces, which is a contradiction.
\end{proof}
\begin{corollary}
For $M$ odd, we have that $[-J_a^{b,\e}]^M=[-J_a^{b,\e}]$.
\end{corollary}
\begin{proposition}\label{022}
Let $cc^M_{\SP}(2n,q)$ denotes the number of $M$-power cyclic
 conjugacy classes in $\Sp$ and 
$cC^M_{\SP}(q,u)=1+\sum\limits_{m=1}^{\infty}cc^M_{\SP}(2m,q)u^{m}$. Then $cC^M_{\SP}(q,u)$ is given by 
\begin{equation}
\left(\dfrac{2}{1-u}-1-u\right)^{h(q,M)}\displaystyle\prod_{d=1}^{\infty}(1-u^d)^{- N_M^*(q,2d)}\prod_{d=1}^{\infty}(1- u^d)^{- R_M^*(q,2d)},
\end{equation}
where \begin{align*}
h(q,M)=\begin{cases}
2 & \text{if }M=\text{odd}\\
1 & \text{otherwise}
\end{cases}.
\end{align*}
\end{proposition}
\begin{proof}
Let $X\in \Sp$ be cyclic. Then $c_X(t)=m_X(t)$. Since the space
$\fq^{2n}$, considered as an $X$-module will be cyclic, 
we have that the primary decomposition of 
$\fq^{2n}$ should be of the form 
$\bigoplus\limits_{f\in\Phi}\dfrac{\fq[t]}{\langle f(t)^a\rangle}$, with each $f\in\Phi$
occuring at most once. Let $\Delta_X=\{(f,\lambda_f):f\in\Phi\}$. Then
$\Delta_X$ represents a cyclic class if and only if
\begin{enumerate}
\itemindent=-13pt
\item $\lambda_{t\pm1}\in2\mathbb{Z}$,
\item $\lambda_f=\lambda_{f^*}\in\mathbb{Z}_{\geq 0}$.
\end{enumerate}
We divide the proof in two cases depending on the value of $(M,q)$.

We start with the case when $(M,q)=1$. In this case using the fact that 
$U_{C_f,n}^M$ is conjugate to $U_{C_f^M}$, we have that 
$X$ is an $M$-th power cyclic polynomial if and only if
\begin{enumerate}
\itemindent=-13pt
\item $\lambda_{t-1}\in 2\mathbb{Z}$,
\item $\lambda_{t+1}\in 2 \mathbb{Z}$ if $M$ is odd and 
$\lambda_{t+1}=0$ if $M$ is even,
\item $(f,\lambda_f)\in\Delta_X$, $f$ is of type $1$ and $\lambda_f\neq 0$,
 then $f\in\Phi^*_M$,
\item $(f,\lambda_f)\in\Delta_X$, $f$ is of type $2$ and $\lambda_f\neq 0$,
 then $f\in\Phi_M\setminus\Phi^*_M$.
\end{enumerate}
We should keep in mind that there are two conjugacy classes corresponding to
the polynomials $t\pm 1$. Hence, we have that
\begin{align*}
cC_\SP^M(q,u)=\left(1+2\sum\limits_{m\geq 1}^\infty
 u^m\right)^{h(q,M)}
\prod\limits_{f\in\Phi^*_M}\left(1+\sum\limits_{m\geq 1}^\infty u^{m\frac{\deg f}{2}}\right)
\prod_{g\in\Phi_M\setminus\Phi_M^*}\left(1+\sum\limits_{m\geq 1}^\infty
u^{ m\deg g}\right)^\frac{1}{2},
\end{align*}
where 
\begin{enumerate}
\itemindent=-13pt
\item the first term accounts for the terms corresponding to $t\pm 1$,
with a power $1$ if $M$ is even and $2$ if $M$ is odd,
\item the second term accounts for polynomial of type $1$ and
\item the third term accounts for polynomial of type $2$, with a power
$\frac{1}{2}$, as for each $g\neq g^*$, the term $(1+\sum\limits_{m\geq 1}^\infty u^{m\deg g})$ occurs 
twice.
\end{enumerate}
Then grouping the polynomials with same degree of type $1$ or $2$, the result follows for the case $(M,q)=1$.
%The similar kind of argument, with taking into account \ref{020} and 
%\ref{021}, we get the result for $(M,q)\neq 1$. 
\end{proof}
\begin{theorem}\label{024}
Let $c^M_{\SP}(n,q)$ denotes the probability of an element to be
 $M$-power cyclic in $\Sp$ and 
$C^M_{\SP}(q,u)=1+\sum\limits_{m=1}^{\infty}c^M_{\SP}(2m,q)u^{m}$. Then $C^M_\SP(q,u)$ is given by
\begin{align}
\left(\dfrac{1}{1-\frac{u}{q}}\right)^{h(q,M)}
\displaystyle\prod_{d=1}^{\infty}
\left(1+\dfrac{u^d}{(q^d+1)(1-\frac{u^d}{q^d})}\right)^{N_M^*(q,2d)}
\prod_{d=1}^{\infty}\left(1+\dfrac{u^d}{(q^d-1)(1-\frac{u^d}{q^d})}\right)^{R_M^*(q,2d)},
\end{align}
if $(q,M)=1$,
where $h(q,M)$ is as in \ref{022}.% and by
%\begin{align}
%\displaystyle\prod_{d=1}^{\infty}\left(1+\dfrac{u^d}{q^d+1}\right)^{N_M^*(q,2d)}\prod_{d=1}^{\infty}\left(1+\dfrac{u^d}{q^d-1}\right)^{R_M^*(q,2d)},
%\end{align} if $(q,M)\neq 1$.
\end{theorem}
\begin{proof}
%For $(M,q)\neq 1$, the result is same as \ref{012}. Hence let us assume $(M,q)=1$. Then 
It follows from Lemmas \ref{002} and \ref{003}, that
\begin{enumerate}
\itemindent=-13pt
\item for $m\geq 2$, the cyclic matrices corresponding to $t\pm 1$,
in $\SP(2m,\fq)$ form two conjugacy classes, with each of the corresponding
centraliser of order $2q^m$,
\item if $f$ is of type $1$ of degree $2m$, then order of the centraliser in $\SP(2ml,\fq)$
of a matrix $X$, with $\Delta_X=\{(f,l)\}$ is $q^{2d(l-1)}(q^d+1)$,
\item if $f$ is of type $2$ of degree $m$, then order of the centraliser in $\SP(2ml,\fq)$
of a matrix $X$, with $\Delta_X=\{(f,l),(f^*,l)\}$ is $q^{2d(l-1)}(q^d-1)$.
\end{enumerate}
Hence using Proposition \ref{022} and the fact that the centraliser of a general block diagonal 
matrix is a direct sum of each of the corresponding centralisers, we have 
the result.
\end{proof}
Analogous statements as in \ref{019}, \ref{022} are true in case of
$\On^\e(n,q)$, whenever $(M,q)=1$.
Hence we consider the case when $(q,2)=1\neq(M,q)$. From \cite{glo},
we know that for unipotent elements of $O^\e(m,q)$ all even Jordan block sizes occur with even multiplicity.
Hence for cyclic $-1$-potent element (i.e. elements $X$ with $c_X(t)=(t+1)^k$), 
we consider unipotent elements which
have odd Jordan block size, with multiplicity $1$. The corresponding 
conjugacy class has representative
\begin{align*}
A_\varepsilon=-\begin{pmatrix}
1&&&&&&&&&\\
1&1&&&&&&&&\\
\vdots&\vdots&\ddots&&&&&&&\\
1&1&\cdots&1&&&&&&\\
1&1&\cdots&1&1&&&&&\\
-\epsilon&-\epsilon&\cdots&-\epsilon&-2\epsilon&1&&&&\\
&&&&&-1&1&&&\\
&&&&&&-1&1&&\\
&&&&&&&\ddots&\ddots&\\
&&&&&&&&-1&1
\end{pmatrix},
\end{align*}
where $\epsilon={1}$ or is a non-square in $\fq$. But then
$A_\varepsilon^M$ is not cyclic $-1$-potent element.
Before writing down the generating functions for the cyclic elements, let us
introduce the following functions:
\begin{define}
We define 
\begin{align*}
Z_\On(u)&=\displaystyle\prod\limits_{d=1}^\infty
\left(1+\dfrac{u^d}{(q^d+1)(1-(\frac{u}{q})^d)}\right)^{N_M^*(q,2d)}
\left(1+\dfrac{u^d}{(q^d-1)(1-(\frac{u}{q})^d)}\right)^{R_M^*(q,2d)},\\
Z'_\On(u)&=\displaystyle\prod\limits_{d=1}^\infty
\left(1-\dfrac{u^d}{(q^d+1)(1+(\frac{u}{q})^d)}\right)^{N_M^*(q,2d)}
\left(1+\dfrac{u^d}{(q^d-1)(1-(\frac{u}{q})^d)}\right)^{R_M^*(q,2d)}.
\end{align*}
\end{define}
\begin{theorem}\label{029}
Let $c_{\On^\e}^M(n,q)$ denotes the probability of an element to be 
$M$-power cyclic in $\On^\e(2n,q)$ with $\e\in\{\pm\}$ and
$c_{\On^0}^M(n,q)$ denotes the probability of an element to be 
$M$-power cyclic in $\On^0(2n+1,q)$. Define 
\begin{align*}
C_{\On^+}^M(q,u)&=1+\sum\limits_{m\geq 1}^\infty
c_{\On^+}^M(m,q)u^m \\  
C_{\On^-}^M(q,u)&=\sum\limits_{m\geq 1}^\infty
c_{\On^-}^M(m,q)u^m \\
C_{\On^0}^M(q,u)&=1+\sum\limits_{m\geq 1}^\infty
c_{\On^0}^M(m,q)u^m. 
\end{align*}  Then
\begin{align}
C_{\On^+}(u^2)+C_{\On^-}(u^2)
+2uC_{\On^0}(u^2)=
\left(1+\dfrac{u^2}{1-\frac{u^2}{q}}\right)^{h(q,M)}Z_\On(u^2),
\end{align}
where $h(q,M)$ is as in \ref{022}.

and
\begin{align}
C^M_{\On^+}(u^2)-C^M_{\On^+}(u^2)=Z'_{O}(u^2).
\end{align}
\end{theorem}
\begin{proof}
We divide the proof in several cases. The first case is when $M$, $q$ are odd.
Then consider the product
\begin{align*}
\left(1+\dfrac{u}{1-\frac{u^2}{q}}\right)^2
\displaystyle\prod\limits_{d=1}^\infty
\left(1+\dfrac{u^{2d}}{(q^d+1)(1-(\frac{u^2}{q})^d)}\right)^{N_M^*(q,2d)}
\left(1+\dfrac{u^{2d}}{(q^d-1)(1-(\frac{u^2}{q})^d)}\right)^{R_M^*(q,2d)}.
\end{align*}
Since $M$ is odd, all of cyclic unipotent or $-1$-potent elements are 
$M$-th power. Now if for a cyclic orthogonal matrix $X$, $c_X(t)$ has 
factor $t\pm 1$, then the multiplicity should be odd. 
There are two conjugacy classes corresponding to each polynomial
$(t\pm 1)^{2l+1}$, with size of centraliser equal to $2q^{l}$. Hence
each of $(t\pm 1)^{2l+1}$, has generating function $1+\frac{2u}{2}+
\frac{2u^3}{2q}+\frac{2u^5}{2q^2}+\cdots=1+\dfrac{u}{1-
\frac{u^2}{q}}$. Hence using arguments similar to \ref{025}, we have that
the product on expansion gives $C_{\On^+}(u^2)+C_{\On^-}(u^2)
+2uC_{\On^0}(u^2)$.

%Next assume that $q$ is even, $(M,q)=1$. Hence $M$ is odd, whence 
%all unipotent element is an $M$-th power. In this case
%the cyclic unipotent matrix can be of order $1\times 1$ or $2m\times 2m$.
%In the first case, order of the centraliser is $1$. In the later case, there
%are two conjugacy classes each having centraliser of order $2q^{m-1}$. Hence in this case $C_{\On^+}(u^2)+C_{\On^-}(u^2)
%+2uC_{\On^0}(u^2)$ is given by
%\begin{align*}
%\left(1+u+\dfrac{u^2}{1-\frac{u^2}{q}}\right)
%\displaystyle\prod\limits_{d=1}^\infty
%\left(1+\dfrac{u^{2d}}{(q^d+1)(1-(\frac{u^2}{q})^d)}\right)^{N_M^*(q,2d)}
%\left(1+\dfrac{u^{2d}}{(q^d-1)(1-(\frac{u^2}{q})^d)}\right)^{R_M^*(q,2d)}.
%\end{align*}

Next suppose $q$ is odd and $M$ is even. Then all the cyclic unipotent
matrices are $M$-th power, where as none of the cyclic $-1$-potent
are $M$-th power. Since unipotent component in cyclic matrices has odd
size, we see that none of the cyclic matrices in $\On^\pm$, has unipotent
part. This along with arguments as before, we have that in this case
$C_{\On^+}(u^2)+C_{\On^-}(u^2)
+2uC_{\On^0}(u^2)$ is given by
\begin{align*}
\left(1+\dfrac{u}{1-\frac{u^2}{q}}\right)
\displaystyle\prod\limits_{d=1}^\infty
\left(1+\dfrac{u^{2d}}{(q^d+1)(1-(\frac{u^2}{q})^d)}\right)^{N_M^*(q,2d)}
\left(1+\dfrac{u^{2d}}{(q^d-1)(1-(\frac{u^2}{q})^d)}\right)^{R_M^*(q,2d)}.
\end{align*}

%For the final case of $(M,q)\neq 1$, since none of the cyclic unipotent or
%$-1$-potent elements are cyclic, we have that, $C_{\On^+}(u^2)+C_{\On^-}(u^2)
%+2uC_{\On^0}(u^2)$ is given by
%\begin{align*}
%\displaystyle\prod\limits_{d=1}^\infty
%\left(1+\dfrac{u^{2d}}{(q^d+1)(1-(\frac{u^2}{q})^d)}\right)^{N_M^*(q,2d)}
%\left(1+\dfrac{u^{2d}}{(q^d-1)(1-(\frac{u^2}{q})^d)}\right)^{R_M^*(q,2d)}.
%\end{align*}
%
For the last equation argument similar to \ref{025} does the job.
\end{proof}

\section{Generating Functions for Regular Matrices}
Since in case of $\Sp$, an element $X\in\Sp$ is regular if and only if
$X$ is cyclic, we concentrate on the case of $\On^\e(m,q)$. We will need the following definition
from \cite{fnp}.
\begin{define}
Let $U$ be a finite dimensional vector space over $\fq$ and 
$X\in\Aut(U,\varphi)$ and $c_X(t)=(t-\mu)^n$ where $\varphi$ is 
an orthogonal form and $\mu=\pm 1$. Then 
 call $X$ to be \textbf{nearly cyclic} if and only if either $U=\{0\}$ or
there is an $X$-invariant orthogonal decomposition $U=U_0\oplus ^\perp U_1$, 
in which $\dim U_0=1$ and $U_1$ is a cyclic $X$-module.
\end{define}
To understand the structure of regular conjugacy classes we state Theorem
$3.2.1$ from \cite{fnp}, which is as follows.
\begin{theorem}\label{030}
Let $q$ be odd and $X\in\On^\e(m,q)$. Then $X$ is regular if and only if
\begin{enumerate}
\itemindent=-13pt
\item for every monic irreducible polynomial $\phi$ other than $t\pm 1$,
the $\phi$-primary component of $X$ is cyclic,
\item for $\mu=\pm1$, the $t-\mu$ component of $X$ is cyclic if it is 
odd dimensional and nearly cyclic if it is even dimensional.
\end{enumerate}
\end{theorem}
\begin{theorem}\label{031}
Assume $q$ to be odd and
Let $r_{\On^\e}^M(n,q)$ denotes the probability of an element to be 
$M$-power regular in $\On^\e(2n,q)$ with $\e\in\{\pm\}$ and
$r_{\On^0}^M(n,q)$ denotes the probability of an element to be 
$M$-power regular in $\On^0(2n+1,q)$. Define 
\begin{align*}
R_{\On^+}^M(q,u)&=1+\sum\limits_{m\geq 1}^\infty
r_{\On^+}^M(m,q)u^m \\  
R_{\On^-}^M(q,u)&=\sum\limits_{m\geq 1}^\infty
r_{\On^-}^M(m,q)u^m \\
R_{\On^0}^M(q,u)&=1+\sum\limits_{m\geq 1}^\infty
r_{\On^0}^M(m,q)u^m. 
\end{align*}  Then
\begin{align*}
R_{\On^+}^M(u)+R_{\On^-}^M(u)+2uR_{\On^0}^M(u)=&
\left(1+\dfrac{u}{1-\frac{u^2}{q}}+\dfrac{qu^2}{q^2-1}
+\dfrac{u^4}{q^2(1-\frac{u^2}{q})}\right)^{h'(M)}\\
&\left(1+\dfrac{u^2}{2(q-1)}+\dfrac{u^2}{2(q+1)}\right)^{h''(M)}Z_{\On}(u^2),
\end{align*}
where $h'(M)=1$ if $M$ is even and $2$ otherwise and
$h''(M)=1$ if $M=2$ and $0$ otherwise.
\end{theorem}
\begin{proof}
We divide the proof in two parts on the basis of parity of $M$ modulo $2$.
Before that, note that 
\begin{align*}
R_{\On^+}^M(u^2)+R_{\On^-}^M(u^2)+2uR_{\On^0}^M(u^2)=
F_1^M(u)F_{-1}^M(u)Z_{\On}(u^2),
\end{align*}
where the functions $F_1^M,F_{-1}^M$ have to be determined. Let us start
with the case of $M$ being odd. Corresponding to the polynomial
$t-1$, the component in the primary decomposition is either cyclic 
(odd dimensional) or nearly cyclic (even dimensional). Both of the cases 
can have $\pm$ types. Note that for odd number $2m+1$, there exist
single conjugacy class of unipotent cyclic elements of $\On^0(2m+1,\fq)$
with centralizer having order $2q^m$ and this is always an $M$-th 
power. There is a single class of nearly cyclic unipotent
elements in $\On^\e(2,\fq)$ consisting of $\mathds{1}$ with order
of the centralizer $2(q-\e 1)$, $\e\in\{\pm\}$. This class is also an $M$ the power. 
For $m\geq 2$, there are two classes of nearly cyclic unipotent
matrices in $\On^\e(2m,\fq)$, $\e\in\{\pm\}$. For each class the corresponding
primary decomposition has one $1$ dimensional space and the other 
being a cyclic $X$-module. In this case these are also $M$-th power. The
centraliser in this case has order $4q^m$.
Hence 
\begin{align*}
F_1^M(u)&=1+\left(\frac{u}{1}+\frac{u^3}{q}+\frac{u^5}{q^2}+\cdots\right)
+\left(\dfrac{u^2}{2(q-1)}+\dfrac{u^2}{2(q+1)}+4\dfrac{u^4}{4q^2}+4\dfrac{u^6}{4q^3}+\cdots\right)\\
&= \left(1+\dfrac{u}{1-\frac{u^2}{q}}+\dfrac{qu^2}{q^2-1}
+\dfrac{u^4}{q^2(1-\frac{u^2}{q})}\right).
\end{align*}
Using same argument and \ref{022},
we find that 
\begin{align*}
F_{-1}^M(u)=\begin{cases} \left(1+\dfrac{u}{1-\frac{u^2}{q}}+\dfrac{qu^2}{q^2-1}
+\dfrac{u^4}{q^2(1-\frac{u^2}{q})}\right)&\text{if } M \text{ odd}\\
1+\dfrac{u^2}{2(q-1)}+\dfrac{u^2}{2(q+1)}&M=2\\
1&\text{otherwise}
\end{cases}
\end{align*}
\end{proof}
\section{Concluding remarks and further question}
\subsection{Existence of root in $\GL(2n,q)$ versus existence of root in $\SP(2n,q)$} Recall from Example 
\ref{2-but-not-2*} that the matrix $A\in \SP(4,5)$ corresponding to the combinatorial data $\{(x^4+3x^3+x^2+3x+1,1)\}$ has a square root $\GL(4,5)$ but not in $\SP(4,5)$. This exhibits an example of a matrix that shows that having a square root (more generally an $M$-th root) in general linear group
does not imply the existence of a square root in symplectic group. 
Hence the notion of $M^*$-power polynomial is different from that of $M$-power polynomial.
\subsection{Closed formula}In the memoir \cite{fnp}, the works are based on using generating functions and getting precise estimates.
These results are one of the great works after that of G. E. Wall and complement the work of Guralnick and Lubeck. 
In the later part of the book, the authors go on finding analytic continuity of the generating functions
beyond the unit disc (with probable poles at $1$ and some few more points). An important ingredient of finding these results heavily relies on one of the famous Rogers-Ramanujan identities, viz.
\begin{align*}
    1+\sum\limits_{n\geq 1}\dfrac{1}{|\GL(n,q)|}=\prod\limits_{\substack{m\geq 1\\m\equiv \pm1\pmod{5}}}\dfrac{1}{1-q^{-m}},
\end{align*}
for proving results about the limiting probabilities in the case of $\GL(n,q)$. To date, analog result is not known to have conclusive results for symplectic and orthogonal groups. 
It will be highly desirable to have closed formula for the generating functions for $M$-th powers in the symplectic and orthogonal groups.

We should also keep in mind that generating functions are used for constructing new modular forms.
Famous examples include explicit formulas for the number of representations of a positive integer as a sum of four and eight squares, whose generating functions are modular forms of weight $2$ and $4$, respectively,
or the partition function $p(n)$, whose generating function is essentially a modular form of weight $-1/2$. It will not be
surprising if the above generating functions give new modular forms and such
a result will be of high interest to a greater audience.
\subsection{Product of $M$-th powers} As mentioned in the introduction we will be happy to draw similar conclusions as 
discussed in \cite{LiObSh12} for finite groups of Lie type, 
at least asymptotically (as $q\longrightarrow\infty$ or $n\longrightarrow\infty$). Our paper is the first step towards the same,
as it sheds light on the scenario for powers in the concerned groups.

%%%%%%%%add the references to the papers here%%%%%%%%%%%%%%%%%%%%%%%%%%%%%%
% \addbibresource{ref.bib} %Import the bibliography file
% \bibliographystyle{siam}
% \bibliography{ref}

\begin{thebibliography}{10}

\bibitem{Bo83}
{\sc A.~Borel}, {\em On free subgroups of semisimple groups}, Enseign. Math.
  (2), 29 (1983), pp.~151--164.

\bibitem{br2}
{\sc J.~R. Britnell}, {\em Cyclic, separable and semisimple matrices in the
  special linear groups over a finite field}, J. London Math. Soc. (2), 66
  (2002), pp.~605--622.

\bibitem{br1}
\leavevmode\vrule height 2pt depth -1.6pt width 23pt, {\em Cycle index methods
  for finite groups of orthogonal type in odd characteristic}, J. Group Theory,
  9 (2006), pp.~753--773.

\bibitem{br3}
\leavevmode\vrule height 2pt depth -1.6pt width 23pt, {\em Cyclic, separable
  and semisimple transformations in the special unitary groups over a finite
  field}, J. Group Theory, 9 (2006), pp.~547--569.

\bibitem{bg}
{\sc T.~C. Burness and M.~Giudici}, {\em Classical groups, derangements and
  primes}, 25 (2016), pp.~xviii+346.

\bibitem{bu}
{\sc M.~C.~R. Butler}, {\em The irreducible factors of {$f(x^m)$} over a finite
  field}, J. London Math. Soc., 30 (1955), pp.~480--482.

\bibitem{f1}
{\sc J.~Fulman}, {\em Cycle indices for the finite classical groups}, J. Group
  Theory, 2 (1999), pp.~251--289.

\bibitem{fnp}
{\sc J.~Fulman, P.~M. Neumann, and C.~E. Praeger}, {\em A generating function
  approach to the enumeration of matrices in classical groups over finite
  fields}, Mem. Amer. Math. Soc., 176 (2005), pp.~vi+90.

\bibitem{fst}
{\sc J.~Fulman, J.~Saxl, and P.~H. Tiep}, {\em Cycle indices for finite
  orthogonal groups of even characteristic}, Trans. Amer. Math. Soc., 364
  (2012), pp.~2539--2566.

\bibitem{f2}
{\sc J.~E. Fulman}, {\em Probability in the classical groups over finite
  fields: {S}ymmetric functions, stochastic algorithms, and cycle indices},
  (1997), p.~148.
\newblock Thesis (Ph.D.)--Harvard University.

\bibitem{glo}
{\sc S.~Gonshaw, M.~W. Liebeck, and E.~A. O'Brien}, {\em Unipotent class
  representatives for finite classical groups}, J. Group Theory, 20 (2017),
  pp.~505--525.

\bibitem{krs1}
{\sc A.~Kulshrestha, R.~Kundu, and A.~Singh}, {\em Asymptotics of the powers in
  finite reductive groups}, Journal of Group Theory,  (2021),
  p.~000010151520200206.

\bibitem{ks}
{\sc R.~Kundu and A.~Singh}, {\em Generating functions for the powers in
  $\text{GL}(n,q)$}, Israel Journal of Mathematics,  (2022), p.~To appear.

\bibitem{La04}
{\sc M.~Larsen}, {\em Word maps have large image}, Israel J. Math., 139 (2004),
  pp.~149--156.

\bibitem{LiObSh10}
{\sc M.~W. Liebeck, E.~A. O'Brien, A.~Shalev, and P.~H. Tiep}, {\em The {O}re
  conjecture}, J. Eur. Math. Soc. (JEMS), 12 (2010), pp.~939--1008.

\bibitem{LiObSh12}
\leavevmode\vrule height 2pt depth -1.6pt width 23pt, {\em Products of squares
  in finite simple groups}, Proc. Amer. Math. Soc., 140 (2012), pp.~21--33.

\bibitem{m}
{\sc H.~Meyn}, {\em Factorization of the cyclotomic polynomial {$x^{2^n}+1$}
  over finite fields}, Finite Fields Appl., 2 (1996), pp.~439--442.

\bibitem{mi}
{\sc J.~Milnor}, {\em On isometries of inner product spaces}, Invent. Math., 8
  (1969), pp.~83--97.

\bibitem{p22}
{\sc S.~Panja}, {\em Powers and skew braces for classical groups},  (2022),
  p.~114.
\newblock Thesis (Ph.D.)--Indian Institute of Science Education and Research
  Pune.

\bibitem{po}
{\sc G.~P\'{o}lya}, {\em Kombinatorische {A}nzahlbestimmungen f\"{u}r
  {G}ruppen, {G}raphen und chemische {V}erbindungen}, Acta Math., 68 (1937),
  pp.~145--254.

\bibitem{Sh09}
{\sc A.~Shalev}, {\em Word maps, conjugacy classes, and a noncommutative
  {W}aring-type theorem}, Ann. of Math. (2), 170 (2009), pp.~1383--1416.

\bibitem{sh}
{\sc K.-i. Shinoda}, {\em The characters of {W}eil representations associated
  to finite fields}, J. Algebra, 66 (1980), pp.~251--280.

\bibitem{ta2}
{\sc D.~E. Taylor}, {\em Conjugacy classes in finite orthogonal groups},
  (2020).

\bibitem{ta1}
\leavevmode\vrule height 2pt depth -1.6pt width 23pt, {\em Conjugacy classes in
  finite symplectic groups},  (2020).

\bibitem{sage}
{\sc {The Sage Developers}}, {\em {S}agemath, the {S}age {M}athematics
  {S}oftware {S}ystem ({V}ersion 0.6.0)},  (2021).
\newblock {\tt https://www.sagemath.org}.

\bibitem{wa}
{\sc G.~E. Wall}, {\em On the conjugacy classes in the unitary, symplectic and
  orthogonal groups}, J. Austral. Math. Soc., 3 (1963), pp.~1--62.

\end{thebibliography}
% \printbibliography

\end{document}